\documentclass[a4paper,twoside,reqno]{amsart}
\usepackage[utf8]{inputenc}

\usepackage{amsmath,amssymb,amsfonts,amsthm,enumerate,mathtools,tikz}

\usepackage{fullpage}
\usepackage{hyperref}
\renewcommand{\MR}[1]{\href{http://www.ams.org/mathscinet-getitem?mr=#1}{MR#1}}

\usepackage{color}

\numberwithin{equation}{section}

\newcommand{\Span}{\mathrm{Span}}
\newcommand{\Nor}{\mathrm{Nor}}

\newcommand{\C}{\mathbb{C}}
\newcommand{\cA}{\mathcal{A}}
\newcommand{\cB}{\mathcal{B}}
\newcommand{\cN}{\mathcal{N}}

\newcommand{\Lcal}{{\mathcal{L}}}
\newcommand{\Rcal}{{\mathcal{R}}}

\newcommand{\asX}{\langle X \rangle}
\newcommand{\asoX}{\langle X \rangle}

\newcommand{\la}{{\triangleright}}
\newcommand{\ra}{{\triangleleft}}

\newcommand{\N}{\mathbb{N}}

\newcommand{\LL}{\mathcal{L}}

\newcommand{\Dcal}{{\mathcal{D}}}
\newcommand{\Ocal}{{\mathcal{O}}}

\newcommand{\Sym}{\mathop{{\rm Sym}}\nolimits}

\renewcommand{\tilde}{\widetilde}

\def\gldim{\operatorname {gl\,dim}}

\def\id{\text{id}}
\theoremstyle{plain}

\newtheorem{thm}{Theorem}[section]
\newtheorem{pro}[thm]{Proposition}
\newtheorem{lem}[thm]{Lemma}
\newtheorem{cor}[thm]{Corollary}

\theoremstyle{definition}

\newtheorem{ex}[thm]{Example}

\newtheorem{convention}[thm]{Convention}

\newtheorem{dfn}[thm]{Definition}
\newtheorem{que}[thm]{Question}
\newtheorem{rmk}[thm]{Remark}

\newcommand{\LM}{\mathbf{LM}}

\newcommand{\extd}{{\rm d}}
\newcommand{\tens}{\otimes}
\renewcommand{\k}{{\bf k}}
\newcommand{\del}{\partial}
\newcommand{\Ann}{{\rm Ann}}

\title{Quadratic algebras associated to permutation idempotent solutions of the YBE}

\author{Tatiana Gateva-Ivanova$^*$}
\address{Max Planck Institute for Mathematics, Vivatsgasse 7, 53111 Bonn, Germany,
and American University in
Bulgaria, 2700 Blagoevgrad, Bulgaria} \email{tatyana@aubg.edu}

\author{Shahn Majid}
\address{School of Mathematical Sciences\\ 327 Mile End Rd, London E1 4NS, UK}
\email{s.majid@qmul.ac.uk}

\thanks{${}^*$ partially supported by the Max Planck Institute for Mathematics,
Bonn,
by the Abdus Salam International Centre for Theoretical Physics, Trieste, and
by Grant KP-06 N 32/1, 07.12.2019 of the Bulgarian National Science Fund.
}

\begin{document}
\date{\today. Ver~1.2}

\begin{abstract} We study the quadratic algebras $\cA(\k,X,r)$ associated to a class of strictly braided but
idempotent set-theoretic solutions $(X,r)$ of the Yang-Baxter or braid relations. In the invertible case, these algebras would be
analogues of braided-symmetric algebras or `quantum affine spaces' but due to $r$ being idempotent they  have very different
properties. We show that all $\cA(\k,X,r)$ for $r$ of a certain permutation idempotent type are isomorphic for a given $n=|X|$,
leading to canonical algebras $\cA(\k,n)$. We study the properties of these both via Veronese subalgebras and Segre products and in
terms of noncommutative differential geometry. We also obtain new results on general PBW algebras which we apply in the
permutation idempotent case.\end{abstract}
\maketitle

\section{Introduction}
\label{Intro}

The linear braid or Yang-Baxter equation (YBE) for a map  $R : V\otimes V \to
V\otimes V$ on a vector space $V$ was extensively studied in the 1980s and solutions lead both to knot invariants in nice case and
to quantum groups, such as the  coquasitriangular bialgebras
$A(R)$ and their Hopf algebra quotients, covariant quantum planes and other structures, see e.g. \cite{MajidQG, FRT,Ma88}. Early
on, V.G. Drinfeld \cite{D}, proposed to also consider the parallel equations for  $r:X\times X\to X\times X$ where $X$ is a set,
and  by now numerous results in this setting have been found, particularly in the involutive case, e.g.
\cite{ESS,GI96,GI04,GI04s,GI12,GIM08,GIM11,GI18,GIVB, Vendramin16, rump}. Non-involutive or strictly braided set-theoretic solutions here are less well understood but of increasing interest, starting with \cite{soloviev, Lu2000}. They have been used to produce knot and virtual knot invariants\cite{nelson} and, more recently, certain non-involutive solutions have been shown to arise from skew braces \cite{Gua_Ven}. Thus, non-involutive solutions  and some of their related algebraic structures have attracted significant further attention, see for instance \cite{Gua_Ven,  ce_smok_vendr19, smok-vendr18, Bachiller18, Vendramin_prob, AL_vendr20, Br19, Colazzo22, Doikou21, doikou_smok21, doikou_smok23, doikou_ryb23, GI21, GIM08, BB23} and references therein.

On the algebra side, we will be particularly interested in quadratic `Yang-Baxter' algebras $\cA(\k , X, r)$  over a field $\k$ proposed in  \cite[Sec 6]{GIM11}  as analogues
of the `quantum planes' in the linear R-matrix theory. In that work, the main results were for $r$ a multipermutation
(square-free) solution of level two. It is also known\cite{GIVB,GI12} that when $X$ is finite  and $r$ is nondegenerate and
involutive then  $\cA(\k,X,r)$ has remarkable algebraic, homological and
combinatorial properties. Although in most cases not even a PBW algebra,  it shares various good
properties of the commutative polynomial ring $\k [x_1, \cdots , x_n]$, namely finite global dimension, polynomial growth,
Cohen-Macaulay, Koszul, and is a Noetherian domain.  More recently, in \cite{AGG} another class of quadratic PBW
algebras called `noncommutative projective spaces' was investigated
and analogues of  Veronese and Segre morphisms between such noncommutative projective spaces were
introduced and studied. In this class, the quadratic relations were almost commutative, allowing the formulation of a relevant
theory of noncommutative Gr\"{o}bner bases. It is natural to formulate similar problems for more general finite solutions $(X,r)$,
but the Yang-Baxter algebras $\cA(\k , X, r)$ in general have complicated quadratic relations which in most cases do not form
Gr\"{o}bner bases. These relations remain complicated even when $\cA(\k,X,r)$ is a PBW algebra, so we need more sophisticated
arguments and techniques, see for example \cite{GI_Veronese, GI23}.

Another starting point for the present work is recent work of  Colazzo et al\cite{Colazzo22,Colazzo23} which introduced a theory of left-nondegenerate  idempotent set-theoretic solutions. These are a particular class of non-involutive solutions which we believe deserve more study. Our main results are for a subclass of `permutation idempotent' solutions and their associated quadratic
algebras first studied in \cite[Prop.~3.15]{Colazzo22}. These depend on a bijection which we denote $f:X\to X$ and have
the form $r_f(x,y)=(f(y),y)$.
 We consider the class $P_n$ of all permutation solutions on a set $X$ with cardinality $|X|= n$ and give an explicit presentation
 of $\cA(\k , X, r_f)$ in terms of generators and explicit $n(n-1)$ quadratic relations which form a reduced Gr\"{o}bner basis, so
 that these algebras are explicitly PBW. We use relations which are different from (and equivalent to) the defining relations
 coming from the original definition of the YB algebra $\cA(\k , X, r_f)$, but with the benefit that the set of new relations forms
 an explicit reduced Gr\"{o}bner basis. We thereby show, remarkably, that all Yang-Baxter algebra $\cA(\k , X, r_f)$ for $r_f$ in
 the class $P_n$  are isomorphic, see Theorem~\ref{thm:YBalgPermSol} and Corollary \ref{cor:main}. Here, the number of \emph{non
 isomorphic} permutation
 solutions $(X, r_f)$ is the number of conjugacy classes in $\Sym(X)$ and hence $p(n)$, the partition function on $n$. This may be
 a large number, all with isomorphic Yang-Baxter algebras. Since we can take $f=\id$, we have moreover a natural representative
 $\cA(\k,n)$ in the isomorphism class for each fixed $n$, which we particularly study.

Further results relate to Veronese subalgebras and Segre products, building on methods for finding these in the context of
Yang-Baxter algebras in \cite{AGG,GI23,GI_Veronese}. The Veronese subalgebra $A^{(d)}$ of a quadratic algebra $A$ is defined as the
subalgebra of elements with degrees that are divisible by $d$. These and related Segre products were previously studied in a
noncommutative setting for general Koszul algebras  by Backelin and Froeberg in \cite{Backelin,Froberg}. A new result here is that
if $(X,r_f)$ is in the class $P_n$ then, for  each integer $d\geq 2$, we construct a new `$d$-Veronese solution' $(W, r_F)$ also in
the class $P_n$, where $W$ is again a set of $n$ elements and $F:W\to W$ is a bijection.  In general, the two solutions $(X,r_f)$
and $(W, r_F)$ are not isomorphic but the latter is constructed so that the $d$-Veronese subalgebra of  $\cA(\k,X,r_f)$ is
$\cA(\k,W,r_F)$, see Theorem~\ref{thm:veronesealg}. Thus,  the class of permutation solutions $P_n$ is closed under taking
$d$-Veronese solutions, in contrast with results on Veronese subalgebras  in \cite{AGG,GI_Veronese}.

For a Segre product of two quadratic algebras $A\circ B$, one needs a quadratic algebra $C$ of a type similar to the type of $A$
and $B$ and an algebra homomorphism $s: C \longrightarrow A \otimes B$, such that the image of $s$ is the putative Segre product
$A\circ B$. One then has to find generators of the kernel to complete the construction. Our result  here, see
Theorem~\ref{thm:rel_segre}, is again that the class of permutation idempotent solutions is closed under Segre products; given two
permutation idempotent solutions in $P_n,P_m$ respectively, we construct a third one in $P_{mn}$, the Yang-Baxter algebra of which
is the Segre product of those of the original two solutions. As well as Veronese subalgebras and Segre products, in
\cite{AGG,GI23} and in the present paper we also consider non-commutative analogues of the Veronese and Segre morphisms, two
fundamental maps that play pivotal roles in classical algebraic geometry \cite{Harris} and in applications to other fields of
mathematics.

A final Section~\ref{secncg} provides some first results on the noncommutative differential geometry of $\cA(\k,n)$ viewed as a
noncommutative version of $\k[x_1,\cdots,x_n]$. We formulate a general construction of first order differential structures
$(\Omega^1,\extd)$ for quadratic algebra and solve for the required data in the case of $\cA(\k,2)$, obtaining a natural
4-parameter family. The same construction works in principle for $\cA(\k,n)$ but with increasingly more solutions. We also provide
a different class of calculi on $\cA(\k,n)$ coming from its bialgebra structure as a monoid algebra. We then round off the paper
with some straightforward computations for  $A(R_f)$, the FRT bialgebra\cite{FRT} for the linear extension of a permutation
idempotent solution $r_f$, under which $\cA(\k,X,r_f)$ is a comodule algebra. We also determine another `fermionic'  Yang-Baxter
comodule algebra for the case where an $R$-matrix $R$ defines an idempotent braiding (just as the standard quantum plane
$\C_q^{2|0}$ has a fermionic partner $\C_q^{0|2}$, see\cite{FRT,MajidQG,Ma88}). By constructions in \cite{MajidQG}, the fermionic
version is necessarily a Hopf algebra in a certain prebraided category (by which we mean that the braiding need not be invertible)
defined by $-R$. Whereas quantum geometry associated to involutive and $q$-Hecke solutions of the Yang-Baxter equations is well
studied, the idempotent case has a very different character as indicated here in the permutation idempotent case.

As well as the key results outlined above,
Section~\ref{sec:graphs} uses PBW and graphical methods to arrive at a general result,  Theorem~\ref{thm:gldiminf}, showing that an
arbitrary $n$-generated PBW algebra $A$ with Gelfand-Kirillov dimension $< n$  has infinite global dimension. Another result,
Theorem~\ref{thm:new}, provides (an exact) lower and an upper bound for the dimension of the grade 2 component of  $\cA(\k, X, r)$
in the case where this is PBW and $(X, r)$ is a left nondegenerate idempotent solution. Equivalently, the theorem provides a lower
and an (exact) upper bound for the  number of relations in the reduced Gr\"{o}bner basis for this Yang-Baxter algebra. The section
ends with Question~\ref{que:conjecture}, an answer to which would characterize all idempotent left nondegenerate solutions of order
$n$ for which the Yang-Baxter algebra has exactly $n(n-1)$ linearly independent quadratic relations forming  a Gr\"{o}bner basis.
Section~\ref{seczero} provides  results on zero divisors  of $\cA(\k,X,r_f)$ and the left annihilator of $\cA(\k,X,r_f)^+$, the
two-sided maximal ideal $(x_1, \cdots, x_n)$ generated by $x_1, \cdots, x_n$. Section~\ref{secmon} contains results on the monoid
$S(X,r_f)$  in the permutation idempotent case. Section~\ref{seq:preliminaries} provides basic algebraic preliminaries for the
paper.

\section{Preliminaries}
\label{seq:preliminaries}
Let $X$ be a non-empty set , and let $\k $ be a field.
We denote by $\asX$
the free monoid
generated by $X,$ where the unit is the empty word denoted by $1$, and by
$\k \asX$-the unital free associative $\k$-algebra
generated by $X$. For a non-empty set $F
\subseteq \k \asX$, $(F)$ denotes the two sided ideal
of $\k \asX$ generated by $F$.
When the set $X$ is finite, with $|X|=n \ge 2$, and ordered, we write $X= \{x_1,
\dots, x_n\},$ and fix the degree-lexicographic order $<$ on $\asX$, where $x_1< \dots
<x_n$.
As usual, $\N$ denotes the set of all positive integers,
and $\N_0$ is the set of all non-negative integers.

We shall consider associative graded $\k$-algebras.
Suppose $A= \bigoplus_{m\in\N_0}  A_m$ is a graded $\k$-algebra
such that $A_0 =\k $, $A_pA_q \subseteq A_{p+q}, p, q \in \N_0$, and such
that $A$ is finitely generated by elements of positive degree. Recall that its
Hilbert function is $h_A(m)=\dim A_m$
and its Hilbert series is the formal series $H_A(t) =\sum_{m\in\N_0}h_{A}(m)t^m$.
For $m \geq 1$,  $X^m$ will denote the set of all words of length $m$ in $\asX$,
where the length of $u = x_{i_1}\cdots x_{i_m} \in X^m$
will be denoted by $|u|= m$.
Then
\[\asX = \bigsqcup_{m\in\N_0}  X^{m},\quad
X^0 = \{1\},\quad   X^{k}X^{m} \subseteq X^{k+m},\]
so the free monoid $\asX$ is naturally \emph{graded by length}. Similarly, the free associative algebra $\k \asX$ is also graded by
length:
\[\k \asX
 = \bigoplus_{m\in\N_0} \k \asX_m,\quad
 \k \asX_m=\k  X^{m}. \]
A polynomial $f\in  \k \asX$ is \emph{homogeneous of degree $m$} if $f \in
\k  X^{m}$.

\subsection{Gr\"obner bases for ideals in the free associative algebra}
\label{sec:grobner}
We remind briefly some basics of the theory of noncommutative Gr\"obner bases, which we use throughout in the paper. In this
subsection $X=
\{x_1,\dotsc,x_n\}$, we fix the degree lexigographic order $<$ on the free monoid $\asX$ extending $x_1 < x_2< \cdots <x_n$ (we
refer to this as {\em deg-lex ordering}).
Suppose $f \in \textbf{ k}\asX$ is a nonzero polynomial. Its leading
monomial with respect to the deg-lex order $<$ will be denoted by
$\LM(f)$.
One has $\LM(f)= u$ if
$f = cu + \sum_{1 \leq i\leq m} c_i u_i$, where
$ c, c_i \in \k $, $c \neq 0 $ and $u > u_i$ in $\asX$, for all
$i\in\{1,\dots,m\}$. Given a set $F \subseteq \k  \asX$ of
non-commutative polynomials, we consider the set of leading monomials
 $\LM(F) = \{\LM(f) \mid f \in F\}.
 $
A monomial $u\in \asX$ is \emph{normal modulo $F$} if it does not contain any of
the monomials $\LM(f)$ as a subword.
 The set of all normal monomials modulo $F$ is denoted by $N(F)$.

Let  $I$ be a two sided graded ideal in $\k\asX$ and let $I_m = I\cap
\k X^m$.
We shall
assume that
$I$ \emph{is generated by homogeneous polynomials of degree $\geq 2$}
and $I = \bigoplus_{m\ge 2}I_m$. Then the quotient
algebra $A = \k  \asX/ I$ is finitely generated and inherits its
grading $A=\bigoplus_{m\in\N_0}A_m$ from $ \k  \asX$. We shall work with
the so-called \emph{normal} $\k$-\emph{basis of} $A$.
We say that a monomial $u \in \asX$ is  \emph{normal modulo $I$} if it is normal
modulo $\LM(I)$. We set
$N(I):=N(\LM(I))$.
In particular, the free
monoid $\asX$ splits as a disjoint union
\begin{equation}
\label{eq:X1eq2a}
\asX=  N(I)\sqcup \LM(I).
\end{equation}
The free associative algebra $\k  \asX$ splits as a direct sum of
$\k$-vector
  subspaces
  \[\k  \asX \simeq  \Span_{\k } N(I)\oplus I,\]
and there is an isomorphism of vector spaces
$A \simeq \Span_{\k } N(I).$

It follows that every $f \in \k \asX$ can be written uniquely as $f =
h+f_0,$ where $h \in I$ and $f_0\in {\k } N(I)$.
The element $f_0$ is called \emph{the normal form of $f$ (modulo $I$)} and denoted
by
$\Nor(f)$.
We define
\[N(I)_{m}=\{u\in N(I)\mid u\mbox{ has length } m\}.\]
In particular, $N(I)_1 =X, N(I)_0=1$. Then
$A_m \simeq \Span_{\k } N(I)_{m}$ for every $m\in\N_0$.

A subset
$G \subseteq I$
of monic polynomials is a \emph{Gr\"{o}bner
basis} of $I$ (with respect to the order $<$) if
\begin{enumerate}
\item $G$ generates $I$ as a
two-sided ideal, and
\item for every $f \in I$ there exists $g \in G$ such that $\LM(g)$ is a
    subword of $\LM(f)$, that is
$\LM(f) = a\LM(g)b$,  for some $a, b \in \asX$.
\end{enumerate}
A  Gr\"{o}bner basis $G$ of
$I$ is \emph{reduced} if (i)  the set $G\setminus\{f\}$ is not a Gr\"{o}bner
basis of $I$, whenever $f \in G$; (ii) each  $f \in G$  is a linear combination
of normal monomials modulo $G\setminus\{f\}$.

It is well-known that every ideal $I$ of $\k  \asX$ has a unique reduced
Gr\"{o}bner basis $G_0= G_0(I)$ with respect to $<$, but, in general, $G_0$ may be
infinite.
 For more details, we refer the reader to \cite{Latyshev, Mo88, Mo94}. The set of leading monomials  of the reduced Gr\"{o}bner
 basis $G_0= G_0(I)$,
\begin{equation}
\label{eq:obstructions}
\textbf{W} =\{LM(f) \mid  f \in G_0(I)\}
\end{equation}
is \emph{the set of obstructions} for $A= \k  \asX/ I$, in the sense of Anick, \cite{Anickmon}.
There are equalities of sets  $N(I) = N(G_0) = N(\textbf{W}).$
\begin{rmk}
\label{rmk:diamondlemma} Bergman's Diamond lemma  \cite[Theorem 1.2]{Bergman} implies the following.
Let $G  \subset \k \asX$  be a set  of noncommutative polynomials. Let $I =
(G)$ and let $A = \k \asX/I.$ Then the following
conditions are equivalent.
\begin{enumerate}
\item
The set
$G$  is a Gr\"{o}bner basis of $I$.

\item Every element $f\in \k \asX$ has a unique normal form modulo $G$,
    denoted by $\Nor_G (f)$.
\item
There is an equality $\cN=N(G) = N(I)$, so there is an isomorphism of vector spaces
\[\k \asX \simeq I \oplus \k \cN.\]
\item The image of $\cN$ in $A$ is a $\k$-basis of $A$, we call it \emph{the normal $\k$-basis of $A$}.
In this case, one can define multiplication $\bullet$ on the $\k$-vector space $\k \cN$ as \[a\bullet b: = \Nor(ab),\quad
\forall a,b \in \k \cN, \]
which gives the structure of a $\k$-algebra on $\k N(G)$  isomorphic to
$A$. We shall often identify $A$ with the $\k$-algebra $(\k N(G), \bullet)$
\end{enumerate}
\end{rmk}

\subsection{Quadratic algebras}
\label{sec:Quadraticalgebras}
A quadratic  algebra is an associative graded algebra
 $A=\bigoplus_{i\ge 0}A_i$ over a ground field
 $\k $  determined by a vector space of generators $V = A_1$ and a
subspace of homogeneous quadratic relations $R= R(A) \subset V
\otimes V.$ We assume that $A$ is finitely generated, so $\dim A_1 <
\infty$. Thus, $ A=T(V)/( R)$ inherits its grading from the tensor
algebra $T(V)$.

In this paper, we consider finitely presented quadratic algebras   $A= \k  \langle X\rangle /(\Re)$,
where by convention $X$ is a fixed finite set of generators of
degree $1$, $|X|=n \geq 2,$ and $(\Re)$ is the two-sided
ideal of relations, generated by a {\em finite} set $\Re$ of
homogeneous polynomials of degree two. In particular, $A_1 = V= \Span_{\k  } X$.
 \begin{dfn}
\label{def:PBW}
A quadratic algebra $A$ is
\emph{a  Poincar\'{e}-Birkhoff-Witt type algebra} or shortly
\emph{a PBW algebra} if there exists an enumeration $X=\{x_1,
\cdots, x_n\}$ of $X,$ such that the quadratic relations $\Re$ form a
(noncommutative) Gr\"{o}bner basis with respect to the
deg-lex order $<$ on $\asX$.
In this case, the set of normal monomials
(mod $\Re$) forms a $\k$-basis of $A$ called a \emph{PBW
basis}
 and $x_1,\cdots, x_n$ (taken exactly with this enumeration) are called \emph{
 PBW-generators of $A$}.
\end{dfn}
 PBW algebras were introduced by Priddy, \cite{priddy} and form an important class of Koszul algebras.
 A PBW basis  is a generalization of the classical
 Poincar\'{e}-Birkhoff-Witt basis for the universal enveloping of a  finite
 dimensional Lie algebra. The interested reader can find information on quadratic algebras and, in
  particular, on Koszul algebras and PBW algebras in \cite{PoPo}.

\subsection{Set-theoretic solutions of the Yang-Baxter equation and their Yang-Baxter algebras}
The notion of \emph{a quadratic set} was introduced in \cite{GI04}, see also \cite{GIM08},  as a
set-theoretic analogue of a quadratic algebra. Here we generalize it by not assuming that the map $r$ is bijective.
\begin{dfn}\cite{GI04}
\label{def:quadraticsets_All} Let $X$ be a nonempty set (possibly
infinite) and let $r: X\times X \longrightarrow X\times X$ be a
 map. In this case,  we refer to  $(X, r)$ as a \emph{quadratic set}. The image of $(x,y)$ under $r$ is
written as
\[
r(x,y)=({}^xy,x^{y}).
\]
This formula defines a ``left action'' $\Lcal: X\times X
\longrightarrow X,$ and a ``right action'' $\Rcal: X\times X
\longrightarrow X,$ on $X$ as: $\Lcal_x(y)={}^xy$, $\Rcal_y(x)=
x^{y}$, for all $x, y \in X$.

(i) $(X, r)$ is \emph{left non-degenerate}, (respecively, \emph{right nondegenerate}) if
the map $\Lcal_x$ (respectively, $\Rcal_x$) is bijective for each $x\in X$.
$(X,r)$ is nondegenerate if both maps $\Lcal_x$ and $\Rcal_x$ are bijective.
(ii) $(X, r)$ is \emph{involutive} if $r^2 = \id_{X\times X}$.
(iii) $(X, r)$ is \emph{idempotent} if $r^2 = r$.
(iv) $(X, r)$ is \emph{a set-theoretic
solution of the Yang--Baxter equation} (YBE) if  the braid
relation
\[r^{12}r^{23}r^{12} = r^{23}r^{12}r^{23}\]
holds in $X\times X\times X,$  where  $r^{12} = r\times\id_X$, and
$r^{23}=\id_X\times r$. In this case, we also refer to  $(X,r)$ as
\emph{a braided set}.
\end{dfn}

\begin{rmk}
    \label{rmk:YBE1}
Let $(X,r)$ be quadratic set.
Then $r$ obeys the YBE,
that is $(X,r)$ is a braided set {\em iff} the following three conditions
hold for all $x,y,z \in X$:
\[
\begin{array}{cccc}
 {\bf l1:}&{}^x{({}^yz)}={}^{{}^xy}{({}^{x^y}{z})},
 \quad\quad
 {\bf r1:}&
{(x^y)}^z=(x^{{}^yz})^{y^z},
 \quad \quad{\rm\bf lr3:}&
{({}^xy)}^{({}^{x^y}{z})} \ = \ {}^{(x^{{}^yz})}{(y^z)}.
\end{array}\]
The map $r$ is idempotent,   $r^2=r$,
\emph{iff}
\[
\textbf{pr:} \quad {}^{{}^xy}{(x^y)}={}^xy,\quad ({}^xy)^{x^y}= x^y, \quad \forall x, y \in X.
\]
\end{rmk}

\begin{convention}
\label{conv:convention1} As a notational tool, we  shall  identify the sets $X^{\times
m}$ of ordered $m$-tuples, $m \geq 2,$  and $X^m,$ the set of all
monomials of length $m$ in the free monoid $\asX$.
Sometimes for simplicity we
 shall write $r(xy)$ instead of $r(x,y)$. \end{convention}

\begin{dfn} \cite{GI04, GIM08}
\label{def:algobjects} To each finite quadratic set $(X,r)$ we associate
canonically algebraic objects generated by $X$ with quadratic
relations $\Re =\Re(r)$ naturally determined as
\[
xy=y^{\prime} x^{\prime}\in \Re(r)\quad \text{iff}\quad
 r(x,y) = (y^{\prime}, x^{\prime})\ \&\
 (x,y) \neq (y^{\prime}, x^{\prime})
\]
 as equalities in $X\times X$. The monoid
$S =S(X, r) = \langle X ; \; \Re(r) \rangle$
 with a set of generators $X$ and a set of defining relations $ \Re(r)$ is
called \emph{the monoid associated with $(X, r)$}. For an arbitrary fixed field $\k$,
\emph{the} $\k$-\emph{algebra associated with} $(X ,r)$ is
defined as
\[\begin{array}{c}
 \cA(\k ,X,r) = \k \langle X  \rangle /(\Re_{\cA})
\simeq \k \langle X ; \;\Re(r)
\rangle;\quad
\Re_{\cA} = \{xy-y^{\prime}x^{\prime}\mid \  xy=y^{\prime}x^{\prime}\in \Re
(r)\}.
\end{array}
\]
Usually, we shall fix an enumeration $X= \{x_1, \cdots, x_n\}$ and extend it to the degree-lexicographic order $<$ on $\asX$.
In this case we require the relations of $\cA$ to be written as
\[
\Re_{\cA} = \{xy-y^{\prime}x^{\prime}\mid \  xy>y^{\prime}x^{\prime}\ \&\  r(xy)= y^{\prime}x^{\prime}\  \text{or}\;
r(y^{\prime}x^{\prime})= xy\}. \]
Clearly, $\cA(\k,X,r)$ is a quadratic algebra generated by $X$
 with defining relations  $\Re_{\cA}$, and is
isomorphic to the monoid algebra $\k S(X, r)$.  When $(X,r)$ is a solution of YBE, we defer to
$\cA(\k,X,r)$ is the associated \emph{Yang-Baxter algebra} (as in \cite{Ma88} for the linear case) or \emph{YB algebra} for short,
and to $S(X, r)$ as the associated Yang-Baxter monoid.
\end{dfn}

If $(X,r)$ is a finite quadratic set then $\cA(\k ,X,r)$ is \emph{a connected
graded} $\k$-algebra (naturally graded by length),
 $\cA=\bigoplus_{i\ge0}\cA_i$, where
$\cA_0=\k ,$
and each graded component $\cA_i$ is finite dimensional.
Moreover, the associated monoid $S(X,r)$ \emph{is naturally graded by
length}:
\begin{equation}
\label{eq:Sgraded}
S = \bigsqcup_{i \geq 0}  S_{i};\quad
S_0 = 1,\; S_1 = X,\; S_i = \{u \in S \mid\;   |u|= i \}, \; S_i.S_j \subseteq
S_{i+j}.
\end{equation}
In the sequel, by `\emph{a graded monoid} $S$', we shall
mean that $S$ is a monoid generated  by $S_1=X$ and graded by length.
The grading of $S$ induces a canonical grading of its monoid algebra
$\k S(X, r).$
The isomorphism $\cA\cong \k S(X, r)$ agrees with the canonical gradings,
so there is an isomorphism of vector spaces $\cA_m \cong \Span_{\k }S_m$.

\begin{rmk}
\label{orbitsinG}
 \cite{GI04s}
 Let $(X,r)$ be a quadratic set, $X= \{x_1, \cdots, x_n\}$ and let $S=S(X,r)$ be the
 associated monoid.

(i) It follows from the defining relations (and the transitive law) that two elements $xy, zt \in X^2$ are equal in $S$ \emph{iff}
\[
r^p(xy)=r^q(zt) \;
\text{is an equality in}\; X^2 \; \text{for some integers}\; p, q \geq 0,\]
where $r^0 = \id$.

 (ii) By definition, two monomials $w_1, w_2 \in \langle X\rangle$ are equal
 in $S$
  \emph{iff} they have equal lengths  $\geq 2$ and there exists a monomial $w_0$ such that each  $w_i, i=1,2$ can be transformed to
  $w_0$ by a finite sequence of
  replacements (they are also called \emph{reductions} in the literature ) each of the form
  \[a(xy)b \longrightarrow a(zt)b,\]
where $xy=zt$ is an equality in $S$, $xy>zt$  in $X^2$ and $a,b \in \langle X\rangle$.

 Clearly, every such replacement preserves monomial length, which therefore
 descends to $S(X,r)$.
Furthermore, replacements coming from the defining relations are possible only
on monomials of length $\geq 2$, hence $X \subset S(X,r)$ is an inclusion.

(ii) It is  convenient for each $m \geq 2$ to refer to the subgroup $D_m=
D_m(r)$ of the braid group $B_m$
generated concretely by the maps
\begin{equation}
\label{eq:r{ii+1}}
r^{ii+1}: X^m \longrightarrow X^m, \  r^{ii+1} =\id_{X^{i-1}}\times r\times
\id_{X^{m-i-1}}, \; i = 1, \cdots, m-1.
\end{equation}
One can also consider the free groups
\[\Dcal_m (r)=  \: _{\rm{gr}} \langle r^{i i+1}\mid i = 1, \cdots, m-1
\rangle,\]
where the
$r^{ii+1}$ are treated as abstract symbols, as well as various quotients
depending on the further type of $r$ of interest.
These free groups and their quotients act on $X^m$ via the actual maps
$r^{ii+1}$, so that
 the image of $\Dcal_m(r)$ in $B_m$ is $D_m(r).$ In particular, $D_2(r) =
 \langle r \rangle$
 is the cyclic group generated by $r$.
 It follows straightforwardly from part (ii) that $w_1,  w_2 \in \langle
 X\rangle$ are equal as elements of $S(X,r)$
  \emph{iff} they have the same length, say $m$, and belong to the same orbit
  $\Ocal_{\Dcal_m}$ of $\Dcal_m(r)$ in $X^m$. In this case, the equality
  $w_1=w_2$ holds in $S(X,r)$ and in the algebra $\cA(\k , X, r)$.

An effective part of our combinatorial approach is the exploration of the
action of the group
$\Dcal_2(r) = \langle r\rangle$ on $X^2$,
and the properties of the corresponding orbits.  In the
literature a $\Dcal_2(r)$-orbit $\Ocal$ in $X^2$ is often called
`\emph{an $r$-orbit}' and we shall use this terminology.
\end{rmk}

 In notation and assumption as above, let $(X,r)$ be a finite quadratic set with
 $S=S(X,r)$ graded by length.
 Then the order of the graded component $S_m$ equals the number of $\Dcal_m(r)$-orbits in $X^m$.

\begin{convention}
\label{rmk:conventionpreliminary1}  Let $(X,r)$ be a finite solution of YBE
of order $n\geq 2$,  and let $\cA = \cA(\k ,X,r)$ be the associated
Yang-Baxter algebra. We fix an arbitrary enumeration
$X=\{x_1, \cdots, x_n\}$ on $X$,
and extend it to the
 deg-lex order $<$ on
$\langle X \rangle$.
By convention, the Yang-Baxter algebra
is presented as
\begin{equation}
\label{eq:Algebra}
\begin{array}{c}
\cA = \k \langle X  \rangle /(\Re_{\cA})
\simeq \k \langle X ; \;\Re(r)
\rangle,\\
\Re_{\cA}
 = \{xy-y^{\prime}x^{\prime}\mid\  xy> y^{\prime}x^{\prime}  \quad \&\quad r(xy)=y^{\prime}x^{\prime} \;\text{or}
 \;r(y^{\prime}x^{\prime}) = xy
 \}.
\end{array}
\end{equation}
Consider the two-sided ideal  $I=(\Re_{\cA})$ of $\k \langle X  \rangle$,
let $G= G(I)$ be the  unique reduced
Gr\"{o}bner basis of $I$  with respect to $<$. Here, we will not need an explicit description of the reduced Gr\"{o}bner basis
$G$  of $I$, but we do need some details.

In general, the set of relations  $\Re_{\cA}$ may not form a Gr\"{o}bner basis of $I$.
However, the shape of the relations  $\Re_{\cA}$ and standard techniques from noncommutative Gr\"{o}bner bases theory imply
that the reduced Gr\"{o}bner basis $G$ is finite, or countably
 infinite, and consists of homogeneous binomials $f_j= u_j-v_j,$ where
 $\LM(f_j)= u_j > v_j$, and $u_j, v_j \in X^m$, for some $m\geq 2$.
The set of all normal monomials modulo $I$ is denoted by $\cN$. As mentioned
above, $\cN= \cN(I) = \cN(G)$.
An element $f \in \k  \asoX $ is in normal form (modulo $I$), if $f \in
\Span_{\k } \cN $.
The free
monoid $\asoX$ splits as a disjoint union
$\asoX=  \cN\sqcup \LM(I).$
The free associative algebra $\k  \asoX$ splits as a direct sum of
$\k$-vector
  subspaces
  $\k  \asoX \simeq  \Span_{\k } \cN \oplus I$,
and there is an isomorphism of vector spaces
$\cA \simeq  \Span_{\k } \cN.$
As usual, we denote
\begin{equation}
\label{eq:Nm}
\cN_{d}=\{u\in \cN \mid u\mbox{ has length } d\}.
\end{equation}
 Then
$\cA_d \simeq \Span_{\k } \cN_{d}$ for every $d\in\N_0$. Note that since the set of relations $\Re_{\cA}$ is a finite set of
homogeneous polynomials,
the elements
of  the reduced Gr\"{o}bner basis $G = G(I)$  of degree $\leq d$ can be found
effectively, (using the standard strategy for constructing a Gr\"{o}bner
basis) and therefore the set of normal monomials $\cN_{d}$  can be found
inductively for $d=1, 2, 3, \cdots .$
It follows from Bergman's Diamond lemma, \cite[Theorem 1.2]{Bergman}, that
if we consider the space $\k  \cN$
endowed with multiplication defined by
 \[f \bullet g := \Nor (fg), \quad \forall\ f,g \in \k  \cN\]
then
$(\k  \cN, \bullet )$
has a well-defined structure of a graded algebra, and there is an isomorphism of
graded algebras
\[
\cA\cong (\k  \cN, \bullet );\quad \cA =\bigoplus_{d\in\N_0}  \cA_d \cong \bigoplus_{d\in\N_0}  \k \cN_d.
\]
By convention, we shall often identify the algebra  $\cA$ with
$(\k \cN, \bullet )$.
Similarly, we consider an operation $\bullet$ on the set $\cN$, with $a
\bullet b:= \Nor(ab)$ for  $a,b \in \cN$,
and identify the monoid $S=S(X,r)$ with $(\cN, \bullet)$, see \cite[Section 6]{Bergman}. \end{convention}

\section{A class of left nondegenerate idempotent solutions and their Yang-Baxter algebras}

Finite idempotent set-theoretic solutions of the Yang--Baxter equation were studied in
\cite{Colazzo22}, where several interesting results were obtained.
We concentrate on a class of concrete such solutions $(X, r_f)$ which we call `permutation idempotent solutions', where $f\in \Sym
(X)$. Such solutions appeared in \cite[Prop.~3.15]{Colazzo22}. In this section we provide new results on the associated Yang-Baxter
algebra $\cA(\k,X,r_f)$. From Definition~\ref{defperm} until the end of the paper, $X$ will be assumed to be of finite order $n
\geq 2$.

\subsection{Left nondegenerate idempotent solutions}

The following proposition shows that for a quadratic set $(X,r)$ of arbitrary cardinality which is left nondegenerate and
satisfies
$r(x,y) = ({}^xy, y)$, for all $x,y \in X$, condition \textbf{l1} in  Remark \ref{rmk:YBE1} is by itself  sufficient to ensure that
$(X,r)$ is a solution
of YBE.

\begin{pro}
\label{pro:important}
Let $X$ be a nonempty set of arbitrary cardinality, and let
$r: X\times X\longrightarrow X\times X$ be a map with the following properties:
\begin{enumerate}
\item $(X,r)$ is left non-degenerate;
\item
\begin{equation}
\label{eq:r(x,y)}
r(x,y) = ({}^xy, y), \;  \; \forall \; x,y \in X.
\end{equation}
\end{enumerate}

Then the following three conditions are equivalent
\begin{enumerate}
\item
\label{solution}
$(X,r)$ is a solution of YBE;
\item
\label{l1}
$(X,r)$ satisfies condition \textbf{l1} in  Remark \ref{rmk:YBE1};
\item
\label{leftequivalent}
There exists a bijection $f \in \Sym(X)$, such that
\[r(x,y) = (f(y), y),\quad\forall\ x,y \in X.\]
\end{enumerate}
In this case $(X, r)$ is an idempotent solution, that is $r^2=r$.
\end{pro}
\begin{proof}
(\ref{l1}) $\Longrightarrow$ (\ref{leftequivalent}).
It follows from \textbf{l1} and  (\ref{eq:r(x,y)}) that
\begin{equation}
\label{eq:l1a}
{}^{(xy)}a  = {}^{x}{({}^ya)}= {}^{{}^xy}{({}^ya)},\quad\forall\  a, x,y \in X.
\end{equation}
Let $t\in X$ be an arbitrary element. By the left nondegeneracy there exists an $a\in X$ such that ${}^ya= t$.
Therefore
\begin{equation}
\label{eq:l1b}
{}^{x}{t}= {}^{{}^xy}{t},\quad\forall\  x,y, t \in X.
\end{equation}
Let $z\in X$ be an arbitrary element. By the left nondegeneracy again, there exists an $y\in X$, such that ${}^xy=z$.
This together with (\ref{eq:l1b})  implies
\begin{equation}
\label{eq:l1c}
{}^{x}{t}= {}^zt, \quad\forall\  x,z, t \in X.
\end{equation}
Therefore $\LL_x = \LL_z,$ for all $x,z \in X$. In particular, there exists a bijection $f \in \Sym(X)$, such that
 $\Lcal_x= f,$ for all $x\in X$. This proves part (\ref{leftequivalent}).

   (\ref{leftequivalent}) $\Longrightarrow$ (\ref{solution})
   Assume (\ref{leftequivalent}). We shall prove that $(X,r)$ is a solution. Let $xyz \in X^3$ be an arbitrary monomial. The
   `Yang-Baxter diagrams',
\begin{equation}
\label{ybediagram}
\begin{array}{l}
 xyz \longrightarrow^{r^{12}} \; f(y)y z\longrightarrow^{r^{23}} \; f(y)f(z) z \longrightarrow^{r^{12}} \;f^2(z)f(z) z\\
 \\
 xyz \longrightarrow^{r^{23}}\; xf(z) z\longrightarrow^{r^{12}}\; f^2(z)f(z) z \longrightarrow^{r^{23}} \;f^2(z)f(z) z\\
 \end{array}
 \end{equation}
show that \[r^{12}r^{23}r^{12} (xyz) =r^{23}r^{12}r^{23}(xyz),\]
for every monomial $xyz\in X^3$, and therefore $(X,r)$ is a solution of YBE.

The implication (\ref{solution}) $\Longrightarrow$  (\ref{l1}) follows straightforwardly from Remark \ref{rmk:YBE1}.
We have proven the equivalence of conditions (1), (2), and (3).

Finally, condition (\ref{leftequivalent}) implies that
the following equalities hold for all $x, y \in X:$
\[r^2(x, y)= r(r(x,y)) = r(f(y), y) = (f(y), y) = r(x,y), \]
which proves that $(X,r)$ is an idempotent solution.

\end{proof}


\begin{dfn}\label{defperm}
Let $X$ be a nonempty set and $f \in \Sym(X)$.
We refer to the left nondegenerate solution $(X,r_f)$ where
\[ r_f: X\times X \longrightarrow X\times X,\quad r_f(x,y) = (f(y), y)\]
as a  \emph{permutation idempotent solution}. We denote by $P_n$ the class of all permutation idempotent solution on sets $X$ of
order $n$ up to isomorphism.
\end{dfn}

\begin{rmk}
\label{rmk:previous_results}
Colazzo at al, \cite{Colazzo22}, studied finite nondegenerate idempotent solutions of YBE and introduced an example in which, by assumption, $(X,r)$ is a finite
solution  of the form
$r(x, y) = (\lambda (x), y)$,  where $\lambda : X \longrightarrow X$ is a permutation. Our starting point in Proposition \ref{pro:important} is significantly  weaker in that we  do not assume that $(X,r)$ is a solution, nor that $X$ is finite. Rather, it includes the result that starting with a quadratic set $(X,r)$ of a certain form and obeying \textbf{l1} leads to a solution with $\Lcal_x$ independent of $x\in X$ and hence given by a permutation.

Next, it was proven in \cite[Prop.~3.15]{Colazzo22} that, in our notation, the Yang-Baxter algebra $\cA(\k,X,r_f)$ for a finite
permutation idempotent solution is a PBW algebra. However, an explicit reduced Gr\"{o}bner basis, and a PBW $\k$-basis (the set of
normal
words modulo the Gr\"{o}bner basis) of the algebra were not found. The Hilbert functions (in particular $\dim \cA_2$) and the
Hilbert series of $\cA$
were also not determined.

Our result, Theorem~\ref{thm:YBalgPermSol} below, takes a different approach and provides  \emph{an
explicit standard
finite presentation} (\ref{eq:rels1}) of $\cA(\k,X,r_f)$, for an arbitrary fixed enumeration $X = \{x_1, \cdots, x_n\}$. This
presentation depends on the enumeration of $X$ but does not depend on the particular permutation $f$, and therefore  all
permutation solutions $(X,r_f)$
share \emph{the same} Yang-Baxter algebra $\cA(\k,X,r_f)$ with PBW-generators $\{x_1, \cdots, x_n\}$, the same
 PBW $\k$-basis  denoted $\cN$, see Corollary \ref{rmk:normalforms}, and the same explicitly given Hilbert series.
More generally, Corollary \ref{cor:main}  implies that all permutation solutions of order $n$ have isomorphic Yang-Baxter algebras.
Our results Theorem~\ref{thm:YBalgPermSol}, Corollary \ref{cor:main}, and Corollary \ref{rmk:normalforms} are crucial for this work
and are used extensively in the paper.
\end{rmk}

\begin{thm}
\label{thm:YBalgPermSol}
  Suppose $(X, r_f)$ is a permutation idempotent solution, where $X= \{x_1, x_2, \cdots, x_n\}$, and $f \in \Sym (X)$.
   By definition, the associated Yang-Baxter algebra  $\cA=\cA(\k,X,r_f)$ has the  presentation
  \begin{equation}
\label{eq:rels0}
\begin{array}{c}
\cA= \k \asX /(\Re_0);\quad \Re_0 = \{x_jx_p- f(x_p)x_p\mid 1 \leq j,p \leq n\},
\end{array}
\end{equation}
where the set $\Re_0$ consists of the binomial relations as shown.  Then
\begin{enumerate}
  \item
  The Yang-Baxter algebra is a PBW algebra with a standard finite presentation
\begin{equation}
\label{eq:rels1}
\begin{array}{c}
\cA= \k \asX /(\Re);\quad \Re = \{x_jx_p- x_1x_p\mid 2 \leq j \leq n, 1 \leq p\leq n\},
\end{array}
\end{equation}
where the set $\Re$ consists of $n(n-1)$ quadratic binomial relations as shown and is the reduced Gr\"{o}bner basis of the two
sided ideal $I= (\Re) = (\Re_0)$ in $\k \asX$.
\item
The set of normal monomials
\begin{equation}
\label{eq:normalbasis}
\cN =\cN(\Re) =\cN(I) = \{1\}\cup\{x_1^{\alpha}x_m\mid \alpha \in \mathbb{N}_0,\ m \in \{1, 2, \cdots, n\}\}
\end{equation}
is a PBW $\k$-basis of $\cA$.
In particular, $\cA$ is a graded algebra of Gelfand-Kirillov dimension $1$  isomorphic to $(\k  \cN, \bullet )$ with
\[
 \cA=\bigoplus_{d\in\N_0}  \cA_d \cong \bigoplus_{d\in\N_0}  \k \cN_d,
\]
where for each $d\geq 1$, the graded component $\cA_d$ has a $\k$-basis
\begin{equation}
\label{eq:N_d}
\cN_d= \{w_1 = x_1^d < w_2 = x_1^{d-1} x_2 < \cdots <w_n = x_1^{d-1} x_n \},
\end{equation}
the set of normal monomials of length $d$.
\item
The Hilbert function $h_\cA$ and  the Hilbert series $H_\cA$ of $\cA$ are
\begin{equation}
\label{eq:hilbert}
h_{\cA}(d) =\dim \cA_d= n,\quad  \forall d \geq 1,\; \quad H_\cA(t) = 1 +nt + nt^2 +nt^3+ \cdots = \frac{n+1 - t}{1-t}.
\end{equation}
\end{enumerate}
\end{thm}
\begin{proof}
The set $X^2$ splits into disjoint $r$-orbits each of which has the shape
\[
\mathcal{O} _p = \{x_1x_p \mapsto_r f(x_p)x_p\mapsto_r f(x_p)x_p, \; x_2x_p \mapsto_r f(x_p)x_p\mapsto f(x_p)x_p, \cdots, x_nx_p
\mapsto_r f(x_p)x_p \mapsto_r f(x_p)x_p\},
\]
where $p\in \{1,  2, \cdots, n\}$.

Any two elements of an $r$-orbit are equal in $\cA$, therefore
\begin{equation}
\label{eq:orbit}
x_1x_p = f(x_p)x_p,\; x_2x_p = f(x_p)x_p,\; \cdots,\;   x_nx_p = f(x_p)x_p
\end{equation}
are equalities, as part of the defining relations of the Yang-Baxter algebra $\cA$.
Observe that $x_1x_p$ is the minimal element in its $r$-orbit $\mathcal{O}_p$, therefore
the set of relations (\ref{eq:orbit}) is equivalent to
the following $n-1$ linearly independent relations
 \begin{equation}
\label{eq:rels2}
x_nx_p- x_1x_p,\quad x_{n-1}x_p - x_1x_p,\quad \cdots, \quad x_2x_p - x_1x_p.
\end{equation}
All additional relations implied by the orbit $\mathcal{O}_p$ are consequences of (\ref{eq:rels2}).
There are exactly $n$ disjoint $r$-orbits $\mathcal{O}_p$, each of which produces exactly $n-1$ relations described in
$\ref{eq:rels2}$.
Every monomial $x_jx_p$,  $2 \leq j \leq n$, occurs exactly once in a relation in (\ref{eq:rels2}). Moreover, it is the leading
monomial
of the relation $x_jx_p -x_1x_p.$ In particular,
\begin{equation}
\label{eq:Normalform}
\Nor (x_jx_p) =x_1x_p,\quad\forall\   2 \leq j \leq n,\  1 \leq p \leq n.
\end{equation}
Thus the set of relations $\Re$ given in (\ref{eq:rels1})
is equivalent to the set of relations in the original definition of the Yang-Baxter algebra $\cA$,
\[\Re_0 = \{xy- f(y) y \mid x,y \in X\}.\]
In particular,
\[\cA = \asX/(\Re).\]
We claim that the set $\Re$ is the reduced Gr\"{o}bner basis of the two sided ideal $I= (\Re)$ in $\k \asX$.
Observe that every ambiguity has the shape $x_kx_jx_i$, where $2 \leq k, j\leq n$. We give the two ways to reduce it using the
relations $\Re$:
\[
\begin{array}{c}
x_kx_jx_i =
x_k(x_jx_i) \longrightarrow x_k(x_1x_i) = (x_kx_1)x_i\longrightarrow x_1x_1x_i \in \cN(\Re),\\
x_kx_jx_i =
(x_kx_j)x_i) \longrightarrow (x_1x_j)x_i) = x_1(x_jx_i)\longrightarrow x_1x_1x_i \in \cN(\Re).
\end{array}
\]
Thus, each ambiguity is solvable and $\Re$ is a Gr\"{o}bner basis of the ideal $I= (\Re)$.
It is then clear that $\Re$ is the reduced Gr\"{o}bner basis of $I$.

The set of leading monomials  of the reduced Gr\"{o}bner basis $\Re$ of $I$ coincides with the set of obstructions
\begin{equation}
\label{eq:obstructionsb}
\textbf{W} =\{LM(f) \mid  f \in \Re \}= \{x_jx_p \mid 2 \leq j \leq n, 1 \leq p \leq n\}.
\end{equation}
There are equalities of sets  $\cN(I) = \cN(\Re) = \cN(\textbf{W})= \cN$. It is obvious that the set $\cN=\cN(\textbf{W})$ of
normal monomials is described explicitly in
(\ref{eq:normalbasis}).
It follows that $\cN_d$, which consists of all normal words of length $d$, has exactly $n$ elements given in (\ref{eq:N_d}), so
$\dim \cA_d = n$, for all $d \geq 1$. Hence,  the algebra $\cA$ has the Hilbert series stated.

It follows from Bergman's Diamond lemma \cite[Theorem 1.2]{Bergman} that
if we consider the space $\k  \cN$
endowed with multiplication defined by
 \[f \bullet g := \Nor (fg),\quad \forall\  f,g \in \k  \cN\]
then
$(\k  \cN, \bullet )$
has a well-defined structure of a graded algebra, and there is an isomorphism of
graded algebras as stated in part (2). \end{proof}

By convention, we shall often identify the algebra  $\cA$ with $(\k \cN, \bullet )$. Figure~\ref{figGamma}(a)
 illustrates some of the above, with on the left a simple 3-element example. The original relations are
\[ x_1x_1=x_2x_1,\quad x_3x_1=x_2x_1,\quad x_1x_2=x_3x_2,\quad x_2x_2=x_3x_2, \quad x_2x_3=x_1x_3,\quad x_3x_3=x_1x_3,\]
while the new (equivalent) relations of $\Re$ are
\[ x_2x_1-x_1x_1,\quad x_3x_1-x_1x_1,\quad x_2x_2-x_1x_2,\quad x_3x_2-x_1x_2,\quad x_2x_3-x_1x_3,\quad x_3x_3-x_1x_3.\]
Figure~\ref{figGamma}(b) shows the corresponding graph of normal words. For a PBW algebra, the graph $\Gamma=\Gamma_{\textbf{N}}$
of normal words
is defined with vertices the generators  $\{x_1, \cdots, x_n\}$ and an arrow $x_i \longrightarrow x_j$ if the word $x_ix_j$ is
normal,  for $1\leq i,j\leq n$. Note in this example there is only one loop. The same applies in general in part (c).
More details and applications of the graph $\Gamma=\Gamma_{\textbf{N}}$ for arbitrary PBW algebras are given in
Section~\ref{sec:graphs}.

\begin{figure}
\[ \includegraphics[scale=0.35]{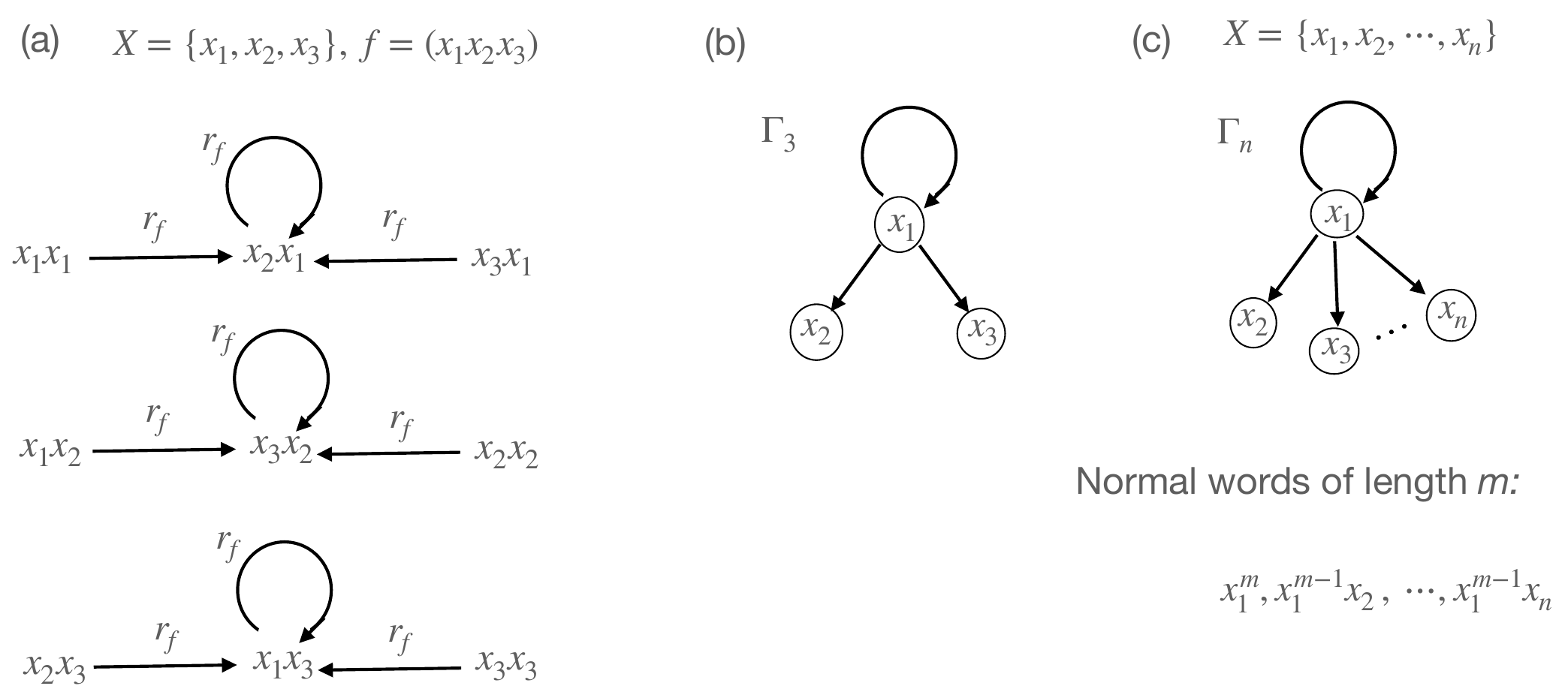}\]
\caption{(a) Graph of $r_f$ orbits in $X^2$ for a set of 3 elements and $f$ a 3-cycle. (b)-(c) Graph of normal words for  3 and in
general $n$ elements.  \label{figGamma}}
\end{figure}

\begin{cor}
\label{cor:main}
Every two permutation idempotent solutions $(X,r_f)$ and $(X,r_g)$, where $f, g \in \Sym(X)$,
have isomorphic Yang-Baxter algebras, $\cA(\k,X,r_f)\cong \cA(\k,X,r_g)$. For a fixed enumeration $X= \{x_1, \cdots, x_n\}$, these
algebras share the same standard finite presentation given in (\ref{eq:rels1}) and the same $\k$-bases $\cN$ of normal words
given explicitly in (\ref{eq:normalbasis}).
\end{cor}

Another consequence of Theorem~\ref{thm:YBalgPermSol} is the following.

\begin{cor}\label{rmk:normalforms} The normal $\k$-basis of $\cA(\k,X,r_f)$ is
\[
\cN = \{1\}\cup\{ x_1^mx_q\mid m \geq 0, 1 \leq q \leq n\}
\]
and the set of normal words of length $d$ is
\[
\cN_d = \{ w_q = x_1^{d-1}x_q\mid 1 \leq q \leq n\}.
\]
Moreover, the equalities
\begin{equation}
\label{eq:Nor2}
y_1y_2\cdots y_{d-1}x_q = (x_1)^{d-1}x_q, \quad\forall\  y_i\in X,\  q \in \{1, \cdots, n\}
\end{equation}
hold in $S(X,r_f)$ and hence in $\cA(\k,X,r_f)$.
\end{cor}
\begin{proof} The form of $\cN$ is from Theorem~\ref{thm:YBalgPermSol}.  Moreover,  every word in
$u \in \asX$ has a unique normal form $\Nor (u)$ (modulo $I$).
It follows from the Diamond Lemma that
\[
\Nor (uv) = \Nor(\Nor(u)\Nor(v)), \quad\forall\  u, v \in \asX.
\]
The  shape of the relations (\ref{eq:rels1}) imply that
\[
\Nor(x_jx_p) = x_1x_p,\quad\forall\ 1 \leq j, p \leq n.
\]
Applying these two rules, and induction on $d$ one yields the following
\begin{equation}
\label{eq:Nor1}
\Nor(y_1y_2\cdots y_{d-1}x_q) = (x_1)^{d-1}x_q,\quad\forall\  y_i\in X,\  q \in \{1, \cdots, n\}.
\end{equation}
 This then implies the equalities stated. \end{proof}

Next, it is proven in \cite[Prop 3.11]{Colazzo22} that the YB-algebra $\cA$ of a finite left non-degenerate idempotent solution is
a finite module over a $\textbf{k}$-subalgebra that is isomorphic to a polynomial algebra in one variable.

As a direct application
of Theorem \ref{thm:YBalgPermSol} and Corollary~\ref{rmk:normalforms}, we now prove a stronger result that the YB-algebra of
an idempotent permutation solution $(X, r_f)$ is a free module of rank $n$ over $\textbf{k}[x_1]$ and give an explicit
$\textbf{k}[x_1]$-basis (a set of free generators) of this free module. Our proof does not make use of \cite[Prop
3.11]{Colazzo22}.

 \begin{pro}
\label{pro:cor:22.2}
The Yang-Baxter algebra $\cA(\k, X, r_f)$ of a finite permutation idempotent solution $(X, r_f)$, where $X= \{x_1, \cdots,
x_n\}$, is a free left module of rank $n$ over the polynomial subalgebra $\k [x_1]$, with free generators $1, x_2, \cdots, x_n$.
 \end{pro}

 \begin{proof}
Let $R= \k [x_1]$ be the subalgebra of $\cA$ generated by $x_1$, it is a commutative polynomial algebra.
Consider the left $R$-module generated by $1, x_2, \cdots, x_n$:
$M =R + \sum_{j=2}^n Rx_j \subseteq \cA$. By the Diamond Lemma, the algebra $\cA$ is identified with the algebra $(\k \cN,
\bullet)$.
More precisely, if $a\in \cA$ is not a constant in $\k $ then we use (\ref{eq:Nor1}) in Corollary~\ref{rmk:normalforms} to find
 a  unique presentation of $a$ as a finite linear combination of normal words in $\cA$:
\[a = a_0 +\sum _{k=1}^q\sum_{j=1}^n \alpha_{kj} x_1^{k-1}x_{j};\quad   a_0,\alpha_{kj} \in \k , \]
 which can be written as
\[a = (a_0+f_1(x_1)x_1)+ f_2(x_1)x_2 + \cdots +f_n(x_1)x_n \in M;\quad f_j(x_1)=\sum _{k=1}^{m_j} \alpha_{kj} x_1^{k-1}, a_0\in
\k .\]
It follows that $\cA =  M$.
Moreover, $1, x_2, \cdots, x_n$ is a set of free generators (a left basis) of $M$ over $\k [x_1]$.
Indeed, assume there is a relation of the form
\[g_1.1 + g_2x_2 + \cdots +g_nx_n= 0;\quad  g_i = \sum_{k=0}^{m_j} \beta_{ik}x_1^k \in  \k [x_1],\quad  1\leq i \leq n.\]
This implies
\begin{equation}
\label{eq:combination}
  \sum_{k=0}^{m_1} \beta_{1k}x_1^k+ \sum_{k=0}^{m_2} \beta_{2k}x_1^kx_2+ \cdots +\sum_{k=0}^{m_n} \beta_{nk}x_1^k x_n = 0,
 \end{equation}
which is a relation involving only distinct monomials from the normal basis $\cN$. This implies that all coefficients
$\beta_{sk}$ occurring in (\ref{eq:combination}) equal zero, and therefore $g_1(x_1)= g_2(x_1)=  \cdots= g_n(x_1)=0$. It follows
that the set
$1, x_2, \cdots , x_n$ is a left basis of the left $\k [x_1]$-module  $M$, so $M$ is a free left $\k [x_1]$-module of rank $n$.
\end{proof}

\begin{rmk} \label{cor:22}  We can
 recover some known properties of nondegenerate idempotent solutions, now by direct application of Theorem \ref{thm:YBalgPermSol} and  Proposition \ref{pro:cor:22.2}. Namely, it is proven in  \cite{Colazzo22}
that the Yang-Baxter algebra $\cA(\k, X, r_f)$ of a finite permutation idempotent solution $(X, r_f)$, where $X= \{x_1, \cdots,
x_n\}$, is
\begin{enumerate}
\item
\label{cor:1}
 Koszul;
 \item
\label{cor:3}  Left Noetherian;
    \item
    \label{cor:4}
     Of Gelfand-Kirillov dimension one;
     \item
     \label{cor:4a} a PI algebra, that is satisfies a polynomial identity;
    \end{enumerate}

In our case, (1) holds because  every PBW algebra is Koszul, see \cite{priddy} and we have seen that $\cA(k,X,r_f)$ is PBW. Similarly, (\ref{cor:3}) and (\ref{cor:4}) follow directly from Proposition \ref{pro:cor:22.2}. It is known that if $ A$ is a finite module over some subalgebra $B$, then
$GK\dim(A)=GK\dim(B)$. Clearly $GK\dim(\k [x_1]) = 1,$ and therefore $GK\dim A = 1$. As an alternative proof one can use directly
the graph of normal words $\Gamma_{\textbf{N}}$ which has the shape in Figure~\ref{figGamma} (c). It has just one cycle
(a loop) passing through the vertex $x_1$,  and therefore by Remark~\ref{rmk:growth},  $GK \dim A = 1$. Part (\ref{cor:4a}) is then
clear since, by   \cite[Thm.~2.2]{SSW85}, each affine (finitely generated) $\k$-algebra $A$ with $GK\dim A = 1$ is PI.
 \end{rmk}

\section{PBW algebras and their associated dual graphs}
\label{sec:graphs}

In this section, we consider general PBW algebras $A$ and investigate the correlation between
the Gelfand Kirillov
dimension $GK\dim A$  and  its global dimension $\gldim A$. To do this, we will make use of a pair of mutually dual graphs associated to
a PBW algebra in \cite[Section 3]{GI12}. We first find  some new properties of the dual graphs and then use further combinatorial arguments to  prove the main results of the section.

Namely, Theorem \ref{thm:gldiminf} is a new result on (general) PBW algebras which proves that an $n$-generated PBW algebra $A$  has
infinite global dimension whenever  $GK\dim A =m
< n$.  Lemma \ref{lem:max_graphgkdim1} gives information about the graphs of normals words for such algebras  with $GK\dim A= 1$
and $n(n-1)$ quadratic relations (or equivalently, $\dim A_2 = \binom{n}{2}+1$). We conclude with Theorem \ref{thm:new} and
Corollary \ref{cor:new} which are new results on YB-algebras $\cA$ of general left nondegenerate idempotent solutions in
the particular case when $\cA$ is PBW.

 Let $A = \k \asX/(\Re)$ be a PBW algebra with a set of PBW-generators $X=\{x_1, \cdots , x_n\}$ $n \geq 2$, where $\Re$ is the
 reduced
Gr\"{o}bner basis of the ideal $I =(\Re)$.
The set $\textbf{W} = \{LM (f) \mid f\in \Re\}$ is called \emph{the set of obstructions} (in the sense of Anick)\cite{Anickhom}.
Then the set of normal words $\cN$ modulo $I$ coincides with the set of normal words modulo $\textbf{W}$, $\cN(I) =
\cN(\textbf{W})= \cN(\Re)$.
In this section,  $\textbf{N}$ will denote the set of normal words of length $2$,
\[\textbf{N} = \cN_2.\]
Note that $X^2$ splits as a disjoint union
\begin{equation}
\label{eq:duality1}
X^2 = \textbf{W} \cup \textbf{N};\quad \textbf{N} = X^2 \setminus \textbf{W},\quad \textbf{W} = X^2 \setminus \textbf{N}.
\end{equation}
Each PBW algebra $A$ has a canonically associated \emph{monomial algebra} $A_{\textbf{W} }= \k \asX/(\textbf{W})$. As
a monomial algebra, $A_{\textbf{W}}$ is also PBW. In fact, the set of monomials  $\textbf{W}$ is a quadratic Gr\"{o}bner basis of
the ideal $J =(\textbf{W})$ with respect to any (possibly new) enumeration of $X$.
Both algebras $A$ and $A_{\textbf{W}}$ have the same set of obstructions $\textbf{W}$
and therefore they have the same normal $\textbf{k}$-basis $\cN$, the same Hilbert series and the same
growth. It follows from results of Anick \cite{Anickhom} that they share the same global dimension
 \[\gldim  A = \gldim  A_{\textbf{W}}.\]
 More generally, the set of
obstructions $\textbf{W}$ determines uniquely the Hilbert series, the growth (Gelfand-Kirillov dimension) and the global dimension
for
the whole family of PBW algebras $A$ sharing the same obstruction set $\textbf{W}$.
In various cases, especially when we are interested in the type of growth or the global dimension of a PBW algebra $A$,  it is more
convenient to work with the corresponding monomial algebra $A_{\textbf{W}}$.

 Each PBW algebra $A$ with a set of PBW-generators $X= \{x_1, \cdots, x_n\}$ and an obstructions set $\textbf{W}$ has two
 associated dual graphs:
 $\Gamma_{\textbf{N}}$, the graph of normal words and $\Gamma_{\textbf{W}}$, the graph of obstructions, see \cite{GI12}, Sec 3 for
 more details. Here we recall some basics.

\begin{dfn}
\label{dfn:graphM}
 Let $M \subset X^2$ be a set of monomials of length $2$. We define the graph $\Gamma_M$
corresponding to $M$ as a directed graph with a set of vertices $V (\Gamma_M) = X$ and a set of directed
edges (arrows) $E = E(\Gamma_M)$ defined as
\[
x \longrightarrow y \in  E\quad \text{iff}\quad  x, y \in X,\  xy \in  M.
\]
Denote by $\widetilde{M}$ the complement $X^2 \setminus M$ and define the `dual' graph $\Gamma_{\widetilde{M}}$ by
$x \longrightarrow y \in  E(\Gamma_{\widetilde{M}})$ \emph{iff} $x,y\in X$ and $x \longrightarrow y$ is not an edge of $\Gamma_M$.
\end{dfn}
Let $A$ be a PBW algebra, let $\textbf{W}$ and $\textbf{N}$ be the set of obstructions and the set of normal
monomials of length 2, respectively.   Then the graph $\Gamma=\Gamma_{\textbf{N}}$ is called \emph{the graph of normal words} of
$A$. Due to
(\ref{eq:duality1}),
the set of obstructions $\textbf{W}$ also determines a graph $\Gamma_{\textbf{W}}$, called \emph{the graph of obstructions} defined
analogously and `dual' to $\Gamma_{\textbf{N}}$ in the above sense.

We recall that \emph{the order of a graph} $\Gamma$ is the number of its vertices, $|V (\Gamma)|$, so $\Gamma_{\textbf{N}}$ is a
graph of
order $|X|$. \emph{A path of length} $k-1$, $k \geq 2$ in $\Gamma$ is a sequence of edges $v_1 \longrightarrow v_2\longrightarrow
\cdots \longrightarrow v_k$, where
$v_i \longrightarrow v_{i+1} \in  E$. \emph{A cycle} of length $k$ in $\Gamma$ is a path of the shape
$v_1 \longrightarrow v_2\longrightarrow \cdots \longrightarrow v_k\longrightarrow v_1$, where $v_1, \cdots , v_k$ are distinct
vertices. \emph{A
loop} is a cycle of length $1$,
$x \longrightarrow x\in E$. So the graph $\Gamma_{\textbf{N}}$
contains a loop $x \longrightarrow x$ whenever $xx \in  \textbf{N}$ and a cycle of length two $x \longrightarrow y \longrightarrow
x$ whenever
$xy, yx \in  \textbf{N}$. In the latter case, $x \longrightarrow y, y \longrightarrow x$ are called \emph{bidirected edges}.
Following
the terminology in graph theory, we distinguish between directed and oriented graphs. A directed
graph having no bidirected edges is \emph{an oriented graph}. An oriented graph without
cycles is \emph{an acyclic oriented graph}.

For example if $(X,r_f)$ is an arbitrary permutation solution on the set $X=\{x_1. \cdots, x_n\}$ and $\cA$ is its YB-algebra, the
set of normal words of length $2$ is $\textbf{N}= \{x_1x_j\mid 1 \leq j \leq n\}$, so the graph $\Gamma_{\textbf{N}}$ has order
$n$, one loop $x_1\longrightarrow x_1$  and exactly $n-1$ additional directed edges $x_1 \longrightarrow x_j$, $2 \leq  j \leq n$
as in Figure~\ref{figGamma} (c).

In general, the graph of normal words  $\Gamma_{\textbf{N}}$ of a given PBW algebra is a directed graph which may contain
bidirected edges, so it is not necessarily an oriented graph. Also observe that a directed graph $\Gamma$ with a set of vertices $V
= \{x_1 \cdots, x_n\}=X$ and a set of directed edges $E(\Gamma)$ determines uniquely a quadratic monomial algebra $A$. Indeed,
consider the set of words $\textbf{N}=\{xy \in X^2 \mid x\longrightarrow y \in E(\Gamma)\}$, and let $\textbf{W}= X^2 \setminus
\textbf{N}$. Then the monomial algebra $A=\k \asX/(\textbf{W})$ has $x_1, \cdots, x_n$ as a set of PBW-generators, $\textbf{W}$ as
a set of obstructions and  $\textbf{N}$ as a set of normal words of length $2$. Moreover,  $\Gamma=\Gamma_{\textbf{N}}$.

The graph of normal words $\Gamma_{\textbf{N}}$
was introduced in a more general context by Ufnarovski and the following remark is a particular case of a more general result of
\cite{ufnarovski}.
\begin{rmk}
\label{rmk:growth}
Let  $A$ be a PBW algebra and let $\cN$ be its set of normal words, with $\textbf{N}= \cN_2$. Then:
\begin{enumerate}\item[(i) ]For every $m\geq 1$, there is a one-to-one correspondence between the set $N_m$ of normal
words of length $m$ and the set of paths of length $m-1$ in the graph $\Gamma_{\textbf{N}}$. The path
$a_1 \longrightarrow a_2 \longrightarrow \cdots \longrightarrow a_m$ (these are not necessarily distinct vertices) corresponds
to the word $a_1a_2 \cdots a_m\in N_m$.
\item[(ii)]  $A$ has exponential growth \emph{iff} the graph $\Gamma_{\textbf{N}}$ has two intersecting cycles.
\item[(iii)] $A$ has polynomial growth of degree $m$ ($GK\dim A=m$) \emph{iff} $\Gamma_{\textbf{N}}$ has no intersecting cycles
    and $m$ is the largest
number of (oriented) cycles occurring in a path of $\Gamma_{\textbf{N}}$.
\end{enumerate}
\end{rmk}

The graph of obstructions $\Gamma_{\textbf{W}}$ can be used to determine explicitly the global dimension of a PBW algebra. The
following result is proven by the first author in \cite[Sec.~3]{GI12}.
\begin{lem} \cite{GI12}
\label{lem:gldim}
A PBW algebra $A$ has finite global dimension $d$ \emph{iff} $\Gamma_{\textbf{W}}$ is an acyclic oriented graph
and $d-1$ is the maximal length of a path occurring in $\Gamma_{\textbf{W}}$.
\end{lem}
\begin{cor}
\label{cor:gldim}
A PBW algebra $A$ has infinite global dimension \emph{iff} the graph of obstructions $\Gamma_{\textbf{W}}$ has a cycle.
\end{cor}
Algorithmic methods for the computation of global dimension of standard finitely presented algebras with polynomial growth in a
more general context were first proposed in \cite{GI88}.

A complete oriented graph $\Gamma$ is called a \emph{tournament} (or a \emph{tour}). In other words, a tournament is a
directed graph in which each pair of vertices is joined by a single edge having a unique direction.
Clearly, a complete directed graph without cycles (of any length) is an acyclic tournament.  The
following is well known in graph theory.
\begin{rmk}
\label{rmk: graph1}
\begin{enumerate}
\item
An
acyclic oriented graph $\Gamma$ with $n$ vertices is a tournament \emph{iff}  $\Gamma$ has exactly $\binom{n}{2}$ (directed)
edges.
\item
 Let $\Gamma$ be an acyclic tournament of order $n$. Then the set of its vertices $V = V (\Gamma)$
can be labeled $V = \{y_1, y_2, \cdots , y_n \}$ so that the set of edges is
\begin{equation}
E(\Gamma ) = \{y_i \longrightarrow y_j \mid  1 \leq i < j \leq n\}.
\end{equation}
Analogously, the vertices can be labeled $V = \{z_1, z_2, \cdots , z_n \}$ so that
\[E(\Gamma) = \{z_j \longrightarrow z_i \mid  n \geq j > i \geq 1\}.\]
\end{enumerate}
\end{rmk}

The proof of the following lemma was kindly communicated by Peter Cameron.
\begin{lem}
\label{lem:peter}
Suppose $\Gamma$ is an acyclic directed graph with a set of vertices $V= \{x_1, \cdots, x_n\}$.  Then $\Gamma$ is a subgraph of an
acyclic tournament $\Gamma_0$ with the same set of vertices.
\end{lem}
\begin{proof}
We claim that one can add new directed edges to connect every two vertices in $V$ which are not connected in such a way that the
resulting graph $\Gamma_0$ is an acyclic tournament. This can be proved by induction on the number of missing edges. So all we have
to do for the inductive step is to add one edge. Suppose that $x, y \in V$ are not joined. Then we claim that we can put an edge
between them without creating a cycle.
Suppose this is false. Then adding $x\longrightarrow y$ would create a cycle $x \longrightarrow y \longrightarrow
u_1\longrightarrow\cdots \longrightarrow u_r \longrightarrow x$, and adding $y\longrightarrow x$ would create a cycle
$y \longrightarrow x \longrightarrow v_1\longrightarrow\cdots \longrightarrow v_s \longrightarrow y.$
 But then there is a cycle
 \[y \longrightarrow u_1\longrightarrow \cdots \longrightarrow u_r \longrightarrow x \longrightarrow v_1\longrightarrow \cdots
 \longrightarrow v_r \longrightarrow y,\]
 contradicting that we start with an acyclic directed graph.
\end{proof}
\begin{lem}
\label{lem:max_graphgkdim1} Let $A= A_{\textbf{W}}$ be a quadratic monomial algebra generated by $X=\{x_1, \cdots, x_n\}$ and
presented as $A_W = \k \langle x_1, \cdots, x_n\rangle/(\textbf{W})$, where $\textbf{W }\subset X^2$ is a set of monomials of
length $2$. Let
$\textbf{N}$ be the set of normal words of length $2$ and assume that $x_1x_1 \in \textbf{N}$ , and that each vertex $x_j$ in the
graph $\Gamma_{\textbf{N}}$ is connected with $x_1$ by a path. The following are equivalent:
\begin{enumerate}
\item The algebra $A$ has Gelfand-Kirillov dimension $GK\dim A=1$ and $\dim A_2 =\binom{n}{2}+1$;
\item The graph $\Gamma_{\textbf{N}}$ is formed out of an acyclic tournament $\Gamma_1$ with vertices $V(\Gamma_1) = X=
    V(\Gamma_{\textbf{N}})$ to which a single loop $x_1\longrightarrow x_1$ is added, so $E(\Gamma_{\textbf{N}})=
    E(\Gamma_1)\bigcup \{x_1 \longrightarrow x_1\}$.
    \item There is a (possibly new) enumeration of $X$, $X = \{y_1\cdots, y_n\}$, such that
   \begin{equation}
    \label{eq:normalwords2}
   \textbf{N} =\{y_iy_j\mid 1 \leq i < j\leq n\}\cup \{yy\}
    \end{equation}
for some fixed $y\in X$.  \end{enumerate}
Moreover, suppose $B$ is a monomial algebra generated by $X= \{x_1, \cdots, x_n\}$ with $GK\dim B =1$, and such that $x_1x_1$ is
a normal word for $B$. Then
\[
\dim B_2 \leq \dim A_2=\binom{n}{2} +1.
\]
\end{lem}
\begin{proof}
Let $\Gamma_1$  be the subgraph of $\Gamma_{\textbf{N}}$ obtained by `erasing' the edge $x_1 \longrightarrow x_1$,
so $E(\Gamma_{\textbf{N}})= E(\Gamma_1)\bigcup \{x_1 \longrightarrow x_1\}$, and $|E(\Gamma_1)|= |E(\Gamma_{\textbf{N}})|- 1$ .
There are equalities
\begin{equation}
\label{eq:numberedges1}
\dim A_2 = |\textbf{N}| = |E(\Gamma_{\textbf{N}})|.
\end{equation}

(1) $\Longrightarrow$ (2). Assume $A$ satisfies (1). Then $GK\dim A=1$ implies that the graph $\Gamma_{\textbf{N}}$  does not
have two cycles connected with a path, or passing through a vertex, see Remark~\ref{rmk:growth}. Moreover, the assumption that
every vertex $x_j$  is connected with $x_1$ by a path implies that the only cycle of $\Gamma_{\textbf{N}}$ is the loop
$x_1\longrightarrow x_1.$ It follows that the subgraph $\Gamma_1$ is an  acyclic directed graph  with exactly $\binom{n}{2}$
edges. Now Remark~\ref{rmk: graph1} (1) implies that $\Gamma_1$ is an acyclic tournament and therefore the graph
$\Gamma_{\textbf{N}}$ has the desired shape.

(2) $\Longrightarrow$ (3). Follows from Remark~\ref{rmk: graph1}, part (2).

(3) $\Longrightarrow$ (1). Assume that after a possible relabeling of the vertices  $X = \{y_1\cdots, y_n\}$ of
$\Gamma_{\textbf{N}}$, the set of edges satisfies (\ref{eq:normalwords2}).
Clearly, $\Gamma_{\textbf{N}}$ has exactly  $\binom{n}{2} +1$ edges, hence $\dim A_2= \binom{n}{2} +1$. Moreover,
$\Gamma_{\textbf{N}}$ contains exactly one cycle  and therefore, by Remark~\ref{rmk:growth}, $GK\dim A =1$.
\end{proof}

Observe that part (1) of the Lemma also holds  if the graph $\Gamma_{\textbf{N}}$ is formed out of an acyclic tournament
$\Gamma_1$ with vertices $V(\Gamma_1) = X= V(\Gamma_{\textbf{N}})$ to which is added an edge $x\longrightarrow z,$ for some $x, z
\in X, x\neq z$,  so  $E(\Gamma_{\textbf{N}})= E(\Gamma_1)\bigcup \{x \longrightarrow z\}$.

\begin{thm}
\label{thm:gldiminf}
If $A$ is a PBW algebra with a set of PBW-generators $x_1, \cdots, x_n$, $n \geq 2$, and $GK\dim A =m < n$, then
$A$ has infinite global dimension, $\gldim A = \infty$.
\end{thm}
\begin{proof}
Consider the graph $\Gamma_{\textbf{N}}$  of normal words. Two cases are possible:

(a) There exists a vertex $x_i\in X$ without a loop $x_i \longrightarrow x_i$ in $\Gamma_{\textbf{N}}$. Then the graph of
obstructions $\Gamma_{\textbf{W}}$
contains the loop $x_i \longrightarrow x_i$ , and therefore, by Corollary~\ref{cor:gldim},  $\gldim A = \infty.$

(b)
The graph $\Gamma_{\textbf{N}}$ contains $n$ loops $x_i \longrightarrow x_i$, $1 \leq i \leq n$, then by Remark~\ref{rmk:growth},
$\Gamma_{\textbf{N}}$ does not have additional cycles (otherwise $A$ would have exponential growth). We shall prove  that there
are two vertices $x, y \in X, x \neq y$ which are not connected with an edge.  Assume on the contrary, that every two vertices
are connected with an edge in $E(\Gamma_{\textbf{N}})$. Consider the subgraph graph $\Gamma_1$ of  $\Gamma_{\textbf{N}}$ obtained
by `erasing' all loops, so
$\Gamma_1$ has set of edges $E(\Gamma_1) = E(\Gamma_{\textbf{N}}) \setminus \{x_i \longrightarrow x_i\mid 1 \leq i \leq n\}$. By
our assumption every two vertices of $\Gamma_1$ are connected with an edge and therefore $\Gamma_1$  is an acyclic tournament of
order $n$.
Then by Remark~\ref{rmk: graph1},
 the set of its vertices $V = V (\Gamma_1)= X$
can be labeled $V = \{y_1, y_2, \cdots , y_n \}$, so that the set of edges is
\begin{equation}
E(\Gamma_1) = \{y_i \longrightarrow y_j \mid  1 \leq i < j \leq n\}.
\end{equation}
This implies that the graph $\Gamma_{\textbf{N}}$ contains a  path with $n$-loops.
\[\includegraphics{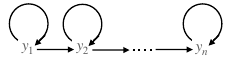}\]
It follows from  Remark~\ref{rmk:growth}  that $GK\dim A \geq n,$ which contradicts the hypothesis $GK\dim A < n$.
Therefore, there are two vertices $x, z \in X, x \neq z$ which are not connected with an edge in $\Gamma_{\textbf{N}}$, so
 the obstruction graph $\Gamma_{\textbf{W}}$
contains the cycle  $x\longrightarrow z\longrightarrow x$. Corollary~\ref{cor:gldim} then implies that $\gldim A = \infty.$
\end{proof}

\begin{cor}
\label{cor:new}
If $(X,r)$ is a finite left nondegenerate idempotent solution of order $|X|=n \ge 2,$ whose Yang-Baxter-algebra
$\cA= \cA(\textbf{k}, X, r)$ is PBW  then the algebra $\cA$ has infinite global dimension, $\gldim A =\infty.$ In particular,
every finite permutation idempotent solution $(X, r_f)$ has Yang-Baxter algebra $\cA$ with  $\gldim A =\infty.$
\end{cor}
\begin{proof}
By \cite[Proposition 3.11]{Colazzo22}, the YB algebra of every idempotent left nondegenerate solution has Gelfand-Kirillov
dimension
$GK\dim \cA = 1$. Therefore, by Theorem \ref{thm:gldiminf}, $\cA$ has infinite global dimension. The YB-algebra of every finite
permutation idempotent solution $(X, r_f)$ is PBW, so $\gldim A =\infty.$
\end{proof}

The following lemma is about general idempotent quadratic sets $(X,r)$. We do not assume any kind of nondegeneracy, nor that
$(X,r)$ is a solution of YBE.
\begin{lem}
\label{lem:2cancel}
Suppose $(X,r)$ is a quadratic set, where $r^2 = r$, and $\cA= \cA(X, \k , r)$, $S = S(X,r)$ are the corresponding quadratic
algebra and monoid.

 \begin{enumerate}
\item If $(X,r)$ is left (resp. right) nondegenerate then the monoid $S$ has left (resp. right) calcellation on monomials of
    length $2$,
that is for all $x,y, z \in X$ there are implications
\[
xy=xz \;\text{in} \; S\; \Longrightarrow y = z\quad (\text{resp.,}\quad   yx= zx \;\text{in} \; S\Longrightarrow y = z).
\]
\item Assume that $(X,r)$ is left nondegenerate and that an enumeration $X = \{x_1, \cdots, x_n\}$ is fixed and, as usual,
    consider the deg-lex ordering on $\asX$. Then the words  $x_1x_1, x_1x_2, \cdots, x_1x_n$  are  normal and distinct in
    $\cA$, hence $\dim \cA_2 \geq n.$
\end{enumerate}
\end{lem}
\begin{proof}
(1) Assume that $xy=xz$ holds in $S$ for some $x,y,z \in X$. Therefore, $xy$ and $xz$ belong to the same $r$-orbit in $X^2$. Two
cases are possible. (a) $r(xy) = ab$ and $r(xz) = ab$ hold in $X^2$ for some $a,b \in X$
or (b) $r(xy) = xz$
(the case $r(zt) =xy$ is analogous). In  case (a), there are equalities in $X^2$
\[r(xy) =({}^xy)(x^y) = ab,\quad r(xz) =({}^xz)(x^z) = ab,\]
which implies ${}^xy = {}^xz = a$. It follows from the left nondegeneracy of $r$ that $y= z$, as claimed.
Now assume case (b). Then $r(xy) =({}^xy)(x^y) = xz$ holds in $X^2$, so ${}^xy= x$. But $r$ is idempotent, so $xz = r(xz)=
({}^xz)(x^z)$ holds in $X^2$. It follows that ${}^xy= {}^xz = x$, which by the left nondegeneracy again implies $y=z$.

(2) For a quadratic algebra $A$, a word $xy\in X^2$ is not normal \emph{iff} $xy- zt$ is in the ideal of relations of $A$, where
$zt \in X^2$ and $xy > zt$ in the deg-lex ordering on $\asX$. It is clear that $x_1x_1\in \textbf{N}$. Suppose $x_1x_j$ is not
normal for some $j>1$, then $x_1x_j - ab$ is in the ideal of relations of $\cA$, where $x_1x_j  >ab$. This implies $a= x_1$, and
$b=x_i$ with $1\leq i < j$. Therefore the equality $x_1x_j = x_1x_i$ holds in $\cA$, so it holds also in $S$. But this is
impossible, since $S$ is 2-cancellative on the left by part (1). It follows that all monomials $x_1x_j, 1 \leq j \leq n$ are
normal.
\end{proof}
\begin{thm}
\label{thm:new}
Suppose $(X,r)$ is a left nondegenerate idempotent solution whose YB-algebra
$\cA$ is PBW with a set of PBW generators $X= \{x_1, \cdots, x_n\}$.
\begin{enumerate}
\item
\label{thm:new1}
There are inequalities
\begin{equation}
\label{eq:dimA2}
n\leq \dim \cA_2  \leq \binom{n}{2}+1.
\end{equation}
Equivalently, the reduced Gr\"{o}bner basis for $\cA$ consists of $N$ linearly independent quadratic binomials, where
\[\binom{n+1}{2}+1 \leq  N = |\textbf{W}| \leq n(n-1).\]
\item
\label{thm:new2}
The lower bound in (\ref{eq:dimA2}) is exact. Moreover,
if $\dim \cA_2= n$ then $\dim \cA_d = n$ for all $d \geq 1.$
\end{enumerate}
\end{thm}

\begin{proof}
(\ref{thm:new1}) It is known that the YB algebra of an idempotent left nondegenerate solution has Gelfand-Kirillov dimension
$GK\dim \cA = 1$, see \cite[Proposition 3.11]{Colazzo22}.

It follows from Lemma~\ref{lem:2cancel} (2) that the words  $x_1x_1, x_1x_2, \cdots, x_1x_n$  are normal and distinct in $\cA$,
hence $\dim \cA_2 \geq n.$
We shall prove that $\dim \cA_2  \leq \binom{n}{2}+1$. Recall that $\dim \cA_2$ equals the number of edges
$|E(\Gamma_{\textbf{N}})|$, so we shall find an upper bound for this number.

Observe that the graph $\Gamma_{\textbf{N}}$ has a loop $x_1\longrightarrow x_1$, and every vertex $x_i$  is connected with $x_1$
by an edge.
Then  Remark~\ref{rmk:growth} and $GK\dim \cA= 1$  imply that the graph $\Gamma_{\textbf{N}}$ has no additional cycles. It
follows that the subgraph $\Gamma_1$ obtained from $\Gamma_{\textbf{N}}$ by `erasing' the loop $x_1\longrightarrow x_1$ is an
acyclic directed graph with a set of vertices $V= \{x_1, \cdots, x_n\}.$
Now Lemma~\ref{lem:peter} implies that $\Gamma_1$ is a subgraph of an acyclic tournament $\Gamma_0$ with the same set of
vertices. Therefore
the number of its edges satisfies the inequality
\[|E(\Gamma_1)|\leq |E(\Gamma_0)|=\binom{n}{2}.\]
But the number of edges of $\Gamma_{\textbf{N}}$ is $|E(\Gamma_{\textbf{N}})|= |E(\Gamma_1)|+1$, and therefore
\[\dim A_2 = |E(\Gamma_{\textbf{N}})| \leq \binom{n}{2}+1.\]
This proves part (\ref{thm:new1}).

(\ref{thm:new2}) We have shown that the YB-algebra $\cA$ of a  permutation idempotent solution $(X,r_f)$ of order $n$  is PBW and
$\dim \cA_2 = n$, so the lower bound is exact, i.e., can be attained.

Suppose now that $(X,r)$ is an arbitrary left nondegenerate idempotent solution such that the YB-algebra $\cA$ is PBW with a set
of PBW-generators  $X= \{x_1, \cdots, x_n\}$.
 Then by Lemma~\ref{lem:2cancel}, each of the monomials $x_1x_j, 1 \leq j \leq n$ is normal.
Therefore $\Gamma_{\textbf{N}}$  contains the loop $x_1 \longrightarrow x_1$ and $n-1$ edges
$x_1\longrightarrow x_j, 2 \leq j \leq n$. If $\dim \cA_2 = n,$ then $\Gamma_{\textbf{N}}$ does not have additional edges, so
$E(\Gamma_{\textbf{N}})= \{x_1\longrightarrow x_j \mid 1 \leq j \leq n\}$. It follows that for each $d \geq 1$,  there are
exactly $n$ distinct paths
of length $d$, namely
\[x_1\longrightarrow x_1\longrightarrow\cdots \longrightarrow x_1 \longrightarrow x_j,\]
which correspond to the normal words $x_1^{d}x_j$, $1 \leq j \leq n$, of length $d+1$. Conversely, by Remark~\ref{rmk:growth},
every normal word of length $d+1$ corresponds to a path of length $d$ in $\Gamma_{\textbf{N}}$. Therefore $|\cN_{d+1}|= n = \dim
\cA_{d+1},$ for all $d \geq 1$.
\end{proof}


 We end the section with some open questions.
\begin{que}
\label{que:conjecture}
\begin{enumerate}
Suppose $(X,r)$ is a left nondegenerate idempotent solution on $X= \{x_1, \cdots, x_n\}$ for which the YB-algebra $\cA$ is PBW
with PBW-generators
the elements of $X$ taken with this fixed enumeration.
\item
Is it true that if $\dim \cA_2 = n$ then $(X,r)$ is a permutation idempotent solution?
\item
What is the exact upper bound for $\dim \cA_2$, i.e. the minimal possible number of relations of $\cA$?
\item More generally, is it true that the permutation idempotent solutions of order $|X| =n$ are the only left nondegenerate
    idempotent solutions for which the lower bound in (\ref{eq:dimA2}) is attained?
\end{enumerate}
\end{que}

\section{The zero divisors in $\cA(\k , X, r_f)$ and the left annihilator of $\cA^+$}
\label{seczero}

In this section $(X, r_f)$ is a permutation idempotent solution on $X = \{x_1, \cdots, x_n\}$,  $\cA = \cA(\k , X, r_f)$ is its
Yang-Baxter algebra and $S= S(X,r_f)$ is its YB-monoid.
Denote by $\cA^+$ the direct sum $\cA^+ = \cA_1\oplus \cA_2\oplus \cA_3 \oplus \cdots$. This is the two-sided ideal $(x_1,
\cdots, x_n)$   generated by $x_1, \cdots, x_n$ and is clearly a left $\cA$-module as well as a maximal left ideal, a maximal
right ideal and a maximal two-sided ideal of $\cA$.
We shall see as a part of Theorem~\ref{thm:kerprod} that  every element of $\cA^+$ is a (right) zero divisor in $\cA$ and,
conversely,  every right zero divisor in $\cA$ belongs to $\cA^+$.
Finally we shall describe $\Ann_\cA(\cA^+)$,
the left annihilator of the left $\cA$-module $\cA^+$.



\begin{cor}
\label{cor:NOcancellation}
We use notation and assumptions as above.
\begin{enumerate}
\item
The following are equalities in the YB-monoid $S(X,r_f)$:
\begin{equation}
\label{eq:importanteq}
\begin{array}{l}
ax_j = bx_j = x_1^{d}x_j,\quad\forall\ 1 \leq j \leq n,\quad  a, b \in S(X,r_f)\ {\rm with}\    |a|= |b|= d;
\end{array}
\end{equation}
\item $S(X,r_f)$ is left cancellative;
\item $S(X,r_f)$ is not right cancellative;
\item The algebra $\cA$ is central, that is its center is the field $\k$.
\end{enumerate}
\end{cor}
\begin{proof}
(1) The equality (\ref{eq:importanteq}) is straightforward from formula (\ref{eq:Nor2}).

(2) Assume that $u.a = u.b$ holds in $S=S(X,r_f)$, for some $a,b,u \in S$. We have to show that $a= b$ holds in $S$.  Without
loss of generality, we can take
$a, b, c \in \cN$. Clearly, $a$ and $b$ have the same length, since $S$ is graded. Then  $a= x_1^d x_i, b = x_1^dx_j$ and $u =
x_1^mx_s$ for some
$1 \leq i,j,s\leq n$ and $0 \leq d,m$.
Hence, by (\ref{eq:importanteq}),
\[u.a = x_1^{m+d}x_i,\quad  u.b =  x_1^{m+d}x_j\]
hold in $S$. This,  together with the equality $u.a = u.b$ in $S$, implies the equality of normal words
$x_1^{m+d}x_i= x_1^{m+d}x_j$, so $i=j$, and hence $a = b$.

(3)  Let $a = x_1^{d},  b =x_1^{d-1}x_2$. Then  $a \neq b$, but by (\ref{eq:importanteq}), one has
 $a.x_q = b.x_q = x_1^dx_q$ for any  $1 \leq q \leq n$. Hence $S$ is not right cancellative.

(4) By \cite[Theorem 3.12]{Colazzo22}, the YB-algebra $\cA(\k,X,r_f)$ is not central \emph{iff} the monoid $S$ is cancellative.
In our case, $S$ is not cancellative, and therefore the center of $\cA$ is the field $\k $.
\end{proof}

\begin{thm}
\label{thm:kerprod}
Let $(X,r_f)$ be a permutation idempotent solution, $X = \{x_1, \cdots, x_n\}$, $\cA=\cA(\k,X,r_f)$ its Yang-Baxter algebra and
$\cA^+ = \cA_1 \oplus \cA_2 \oplus \cdots$.
\begin{enumerate}
\item
If $a,b \in \cA$ and $a, b \neq 0$ with $ab = 0$ then $a, b \in \cA^+$ and we have a presentation
\begin{equation}
\label{eq:M1}
a = \sum_i \alpha_{1i}x_i + \sum_i \alpha_{2i}x_1x_i +\cdots +  \sum_i \alpha_{pi}(x_1)^{p-1}x_i,
    \quad p \geq 1,\ \alpha_{di}\in \k ,\  1\leq i \leq n, 1 \leq d \leq p,
\end{equation}
where  $\sum_i \alpha_{di} = 0$ for all $1 \leq d \leq p$.
\item Conversely, if $a\in \cA$ satisfies (\ref{eq:M1}) then  $ab= 0$ for all $b \in \cA^+$.
\end{enumerate}
\end{thm}
\begin{proof}
Suppose $a \in \cA$, $a \neq 0$ and assume that there exists  $b \in \cA$, $b \neq 0$, such that $ab=0$. It is clear that each
$a, b \in \cA_1 \oplus \cA_2 \oplus \cA_3 \oplus \cdots$, since $\k $ is a field.
Suppose $b \in A_m \oplus A_{m+1} \oplus \cdots, $  $m\geq 1$, where the first nonzero graded component, $b_m$, of $b$ is
\begin{equation}
\label{eq:gm}
b_m = \sum_j \beta_{mj} (x_1)^{m-1}x_j; \quad (\beta_{m1}, \cdots, \beta_{mn})\neq (0, \cdots, 0).
\end{equation}
In the case when $m=1$ we have simply $x_1^{m-1}=x_1^0= 1$.
The elements $a$ and $b$ have presentations
\[
a= \sum_i \alpha_{1i}x_i + \sum_i \alpha_{2i}x_1x_i +\cdots +  \sum_i \alpha_{pi}(x_1)^{p-1}x_i,\]
where $p \geq 1,\ \alpha_{di}\in \k,\  1\leq i \leq n,\  1 \leq d \leq p$, and
\[b =\sum_j \beta_{mj} (x_1)^{m-1}x_j + \sum_j \beta_{(m+1)j} (x_1)^{m}x_j +\cdots +  \sum_j \beta_{qj}(x_1)^{q-1}x_j,\]
where $1 \leq m \leq q,\ \beta_{dj}\in \k,\ m \leq d \leq q,\ 1\leq j \leq n$.  We shall use induction on $d$ to prove that
\begin{equation}
\label{eq:M2a}
\sum_i \alpha_{di} = 0,\quad\forall\ 1 \leq d \leq p.
 \end{equation}
The equality $a.b= 0$ implies that each graded component $(a.b)_d= 0$. In fact, since $b_k = 0,$ for $k \leq m-1$ the first
graded component of the product is
\begin{equation}
\label{eq:M3}
\begin{array}{ll}
(ab)_{m+1}&= a_1 b_m = (\sum_i \alpha_{1i}x_i)(\sum_j \beta_{mj} (x_1)^{m-1}x_j)\\&\\
         &=\sum_i \alpha_{1i} \sum_j \beta_{mj} x_i.(x_1)^{m-1}x_j)=\sum_i \alpha_{1i} \sum_j \beta_{mj} (x_1^{m}x_j)= 0.
\end{array}
\end{equation}
   For the last equality, we use that $y_1 \cdots y_m x_j = (x_1)^{m}x_j$, for all $y_1, \cdots, y_m \in X$ and all $m \geq 1$,
   see formula (\ref{eq:Nor1}).
 We obtain a linear relation
 \[\sum_i \alpha_{1i} \sum_j \beta_{mj} (x_1^{m}x_j)= 0\]
 for the linearly independent monomials  $x_1^{m}x_j, 1 \leq j \leq n$, and therefore
 \begin{equation}
\label{eq:M4}
(\sum_i \alpha_{1i}) \beta_{mj} = 0,\quad\forall\ 1 \leq j \leq n.
 \end{equation}
 By assumption, $b_m\neq 0$. Hence, there exists a $j, 1 \leq j \leq n,$ such that $\beta_{mj} \neq 0$, so
 (\ref{eq:M4}) implies the desired equality
\begin{equation}
\label{eq:M5}
\sum_i \alpha_{1i}  = 0.
 \end{equation}
 This gives the base for the induction. Next, we take:

 Induction Hypothesis (IH): Assume (\ref{eq:M2a}) holds for $1 \leq d \leq k-1$.
 Consider the $(m+k)$-th component
 \[(a.b)_{m+k}= a_kb_{m}+ a_{k-1}b_{m+1}+ \cdots +a_1 b_{m+k-1}=0.\]
 More precisely, one has
\begin{align*}
 (\sum_i &\alpha_{ki}( x_1)^{k-1}x_i)(\sum_j \beta_{mj} (x_1)^{m-1}x_j)+ \cdots + (\sum_i \alpha_{1i} x_i)(\sum_j
 \beta_{(m+k-1)j} (x_1)^{m+k-2}x_j)\\
  &=(\sum_i \alpha_{ki}\sum_j \beta_{mj}+ \sum_i \alpha_{(k-1)i}\sum_j \beta_{(m+1)j} +\cdots +\sum_i \alpha_{1i}\sum_j
  \beta_{(m+k-1)j})  x_1^{k+m-1}x_j= 0.
 \end{align*}
Then, since for fixed $k$ and $m$,  the monomials $x_1^{k+m-1}x_j$ for $1 \leq j \leq n$ are linearly independent, we have
\begin{equation}
\label{eq:M6}
(\sum_i \alpha_{ki})\beta_{mj}+ (\sum_i \alpha_{(k-1)i})\beta_{(m+1)j} +\cdots +(\sum_i \alpha_{1i})\beta_{(m+k-1)j}= 0
\end{equation}
for each fixed $1\leq j \leq n$. By the IH, we can assume that
\[
\sum_i \alpha_{1i} = 0,\quad  \cdots,\quad  \sum_i \alpha_{(k-1)i} =0,
\]
so that (\ref{eq:M6}) implies
\[
(\sum_i \alpha_{ki})\beta_{mj}= 0,\quad\forall\ 1 \leq j\leq n.
\]
But we know that there exists a $j, 1 \leq j \leq n,$ such that $\beta_{mj} \neq 0$, see (\ref{eq:gm}). Hence,  $\sum_i
\alpha_{ki}= 0$ as desired.
This proves the first part of the theorem. Direct computation shows that, conversely, if $a\in \cA$ satisfies (\ref{eq:M1}) then
$a. \cA =0$.
\end{proof}

Recall that the left annihilator of a nonempty subset set $S\subseteq A$ of an algebra $A$ is defined as
\[\Ann_A(S) = \{a \in A \mid  ab = 0, \ \forall\ b \in S \}.\]
Clearly, the left annihilator  is a left ideal of $A$.
The right annihilator $\Ann^R_A(S)$ is defined analogously.
It is obvious that either annihilator of a unital algebra is the zero ideal $\{0\}$. It follows from Theorem~\ref{thm:kerprod}
that the right annihilator of $\cA^+$ is also the zero ideal,
$\Ann^R_\cA(\cA^+) = \{0\}$.

\begin{pro}
\label{pro:annihilator}
In the setting of Theorem~\ref{thm:kerprod}.
Let $(X,r_f)$ be a permutation idempotent solution, $X = \{x_1, \cdots, x_n\}$, $\cA=\cA(\k,X,r_f)$ its Yang-Baxter algebra and
$\cA^+ = \cA_1 \oplus \cA_2 \oplus \cdots$.
The left annihilator $\Ann_{\cA}(\cA^+)$ is the left ideal generated by the elements $(x_i-x_{i+1}), \; 1 \leq i \leq n-1.$
Moreover, it is a free left module of rank $n-1$ over the commutative polynomial ring $\k [x_1]$ with a free left basis
$\{x_i-x_{i+1}, \; 1 \leq i \leq n-1\},$
\begin{equation}
\label{eq:Ann}
\Ann_{\cA}(\cA^+)= \bigoplus_{1 \leq i\leq n-1} \k [x_1](x_i-x_{i+1}).
\end{equation}
\end{pro}
\begin{proof}
We let $\cA=\cA(\k,X,r_f)$ and observe first that
\begin{equation}
\label{eq:f1}
a= \sum_{i=1}^n \alpha_i x_i;\quad  \sum_i \alpha_i = 0,\quad \alpha_i \in \k , 1 \leq i \leq n
\end{equation}
if and only if
\begin{equation}
\label{eq:f2}
a= \sum_{j=1}^{n-1} \beta_j (x_j - x_{j+1});\quad  \beta_j  \in \k , 1 \leq j \leq n-1.
\end{equation}
More specifically, given $a$, each of the presentations (\ref{eq:f1}) and (\ref{eq:f2}) determines uniquely the second
presentation via the formulae
\begin{equation}
\label{eq:f3}
\beta_1 = \alpha_1, \;\;
\beta_k =\alpha_k + \alpha_{k-1} +\cdots + \alpha_2- \alpha_1,\;\; 2\leq k \leq n-2, \;\; \beta_{n-1}= -\alpha_n.
\end{equation}
It is clear that each of the elements $x_i-x_{i-1}, 1 \leq i \leq n-1$  is in the left annihilator
 $\Ann_{\cA}(\cA^+)$, and therefore, the left ideal $\sum_{i = 1}^n \cA(x_i-x_{i+1}) \subseteq \Ann_{\cA}(\cA^+)$.
 It follows from Theorem~\ref{thm:kerprod} that the graded components of each $a\in \Ann_{\cA}(\cA^+)$
 satisfy
 \[a_d= \sum_i \alpha_{di}x_1^{d-1}x_i=(x_1^{d-1})\sum_i \alpha_{di}x_i=(x_1^{d-1})\sum_{j=1}^{n-1} \beta_{dj}(x_j-x_{j+1}),\]
 where $\sum_{i=1}^n \alpha_{di}=0$ and the coefficients $\beta_{dj}, 1 \leq j\leq n-1$ are expressed via $\alpha_{di}, 1 \leq i
 \leq n$ using the formulae (\ref{eq:f3}). Hence, there is an equality of ideals
 \[\Ann_{\cA}(\cA^+) = \sum_{i=1}^{n-1} \cA(x_i-x_{i+1}).\]
 We know that $\cA$ is a free left module with basis $1, x_2, \cdots, x_n$  over the polynomial algebra $\k [x_1]$, so
 \begin{equation}
 \label{eq:A}
 \cA = \k [x_1]\oplus \k [x_1] x_2 \oplus \cdots \oplus \k [x_1] x_n.
 \end{equation}
 Keeping in mind that
 \[
 x_i(x_j-x_{j+1}) = x_1(x_j-x_{j+1})
 \]
 holds in $\cA$ for all $1 \leq i \leq n, 1 \leq j \leq n-1$,
 we obtain that
 \[
 \Ann_{\cA}(\cA^+) = \sum_{j=1}^{n-1} \k [x_1] (x_j-x_{j+1}).
 \]
 To  prove that this is a direct sum, one uses (\ref{eq:A}).
 Assume that
 \[\sum _{j = 1}^{n-1}a_j(x_j-x_{j+1}) = 0,\quad a_j  \in \k [x_1],\quad 1 \leq j \leq n-1.  \]
 Then
 \[a_1x_1 + (a_2-a_1)x_2+ (a_3-a_2)x_3 +\cdots +(a_{n-1}-a_{n-2})x_{n-1} - a_{n-1}x_n = 0,\]
 which is a relation for the left basis  of the free left $\k [x_1]$-module $\cA$.
 It follows that $a_1 =a_2 =\cdots = a_{n-1}= 0.$
\end{proof}

Next, a result in \cite[Corollary 7.4]{Colazzo23} is that for an arbitrary  finite left
nondegenerate braided set with YB-algebra $\cA$,  the Jacobson radical ${\rm Jac}(\cA)$ coincides with $B(\cA)$, the lower
nilradical
of $\cA$. We now recover a similar result but by a different route
 as an application of properties of $\cA$ and $\cA^+$ above
and some classical results from ring theory.

\begin{cor} cf.\cite{Colazzo23}
\label{pro:annihilator-Jacobson}
Let $(X,r_f)$ be a permutation idempotent solution, $X = \{x_1, \cdots, x_n\}$, $\cA=\cA(\k,X,r_f)$ its Yang-Baxter algebra and
$\cA^+ = \cA_1 \oplus \cA_2 \oplus \cdots$.
The Jacobson radical ${\rm Jac}(\cA)$ of $\cA$ is nilpotent and coincides with $\Ann_{\cA}(\cA^+)$.
\end{cor}
\begin{proof}
We shall give a direct proof using the properties of $\Ann_{\cA}(\cA^+)$ and the Braun-Kemer-Razmyslov
 Theorem\cite{BelovRowen20}.
Recall that the \emph{upper nilradical} ${\rm Nil}(R)$ of a unital ring $R$ is defined as the ideal generated by all nil ideals
of the ring, and is itself a nil ideal. The \emph{Jacobson radical} ${\rm Jac}(R)$ of a unital ring $R$ can be defined as the
unique left ideal of $R$ maximal with the property that every element $r \in {\rm Jac}(R)$ is left quasiregular (or,
equivalently, right quasiregular), i.e., $1 -r$ is a unit of $R$. It is well known that ${\rm Nil}(R) \subseteq {\rm Jac}(R)$ for
any unital ring $R$.

 We shall prove first that $\Ann_{\cA}(\cA^+)$ coincides with the upper nil radical of $\cA$.
Indeed, every $a\in \Ann_{\cA}(\cA^+)$ satisfies $a^2 = 0$, therefore $\Ann_{\cA}(\cA^+)$ is a nil ideal.
  Moreover, every nil element of $\cA$ is in $\Ann_{\cA}(\cA^+)$, for if $g\in \cA$ satisfies $g^{m-1}\neq 0, g^m=0$, where $m
  \geq 2$,
 then $g, g^{m-1} \in \cA^+$ and by Theorem~\ref{thm:kerprod} the equality $g.(g^{m-1})= 0,$ together with $g^{m-1}\neq 0$ imply
 that $g \in \Ann_{\cA}(\cA^+)$. In particular, $\Ann_{\cA}(\cA^+)$ contains every nilpotent ideal $I$ of $\cA$. It follows that
 \[\Ann_{\cA}(\cA^+) = {\rm Nil} (\cA).\]

The Braun-Kemer-Razmyslov Theorem states that the Jacobson radical of any affine (i.e., finitely generated) PI algebra over a
field
$\k$ is nilpotent\cite[Theorem 1.1]{BelovRowen20}. But $\cA$ is affine and PI, and therefore ${\rm Jac}(\cA)$ is nilpotent. It
follows that
${\rm Jac}(\cA) \subseteq \Ann_{\cA}(\cA^+) = {\rm Nil}(\cA)$, which together with the well-known inclusion  ${\rm Jac}(\cA)
\supseteq {\rm Nil}(\cA)$ implies  ${\rm Jac}(\cA) = \Ann_{\cA}(\cA^+) = {\rm Nil}(\cA)$.
\end{proof}

\section{$d$-Veronese solutions, subalgebras and morphisms for permutation idempotent solutions}
\label{sec:merged}
In this section, we first introduce certain `$d$-Veronese solutions' associated with an arbitrary  braided set $(X,r)$. We find
these explicitly in the permutation idempotent case  and then use them to construct the $d$-Veronese subalgebras and Veronese
morphisms for the associated Yang-Baxter algebras $\cA(\k,X,r_f)$. Theorem~\ref{thm:veronesealg} presents these $d$-Veronese
subalgebras in terms of generators and quadratic relations and shows that they are all isomorphic to the original Yang-Baxter
algebra. The general strategy here follows the lines of \cite{GI_Veronese} for $d$-Veronese subalgebras and $d$-Veronese
morphisms $v_{n,d}$ for the Yang-Baxter algebras of various finite braided sets $(X,r)$. However, due to great difference between
the properties
of permutation idempotent solutions compared to the cases in  \cite[ Thm.~4.12, Thm.  5.4]{GI_Veronese}, our results here are
significantly different.
\subsection{Braided monoids and the $d$-Veronese solution}
\label{secmon}
Here, we introduce new solutions (braided sets) associated naturally with a given braided set $(X,r)$ and its braided monoid
$S(X,r)$. Matched pairs of monoids, M3-monoids and braided monoids in a general
setting were studied in \cite{GIM08}, where the interested reader can find the
necessary definitions and
the original results. Here we extract only the following facts from \cite[Thm. 3.6, Thm. 3.14]{GIM08} and their proofs.

Let  $(X,r)$ be a braided set and $S=S(X,r)$ its Yang-Baxter monoid.
Then
\begin{enumerate}
\item
The left and
right actions
$
{}^{(\;\;)}{\circ}: X\times X  \longrightarrow
 X $ and $\circ^{(\;\;)}: X
\times X \longrightarrow  X
$
defined via $r$ can be extended in a unique way to left and
right actions
\[{}^{(\;\;)}{\circ}: S\times S  \longrightarrow
 S,\; \; (a, b) \mapsto  {}^ab,\quad  \circ^{(\;\;)}: S
\times S \longrightarrow  S, \;\; (a, b) \mapsto  a^b
\]
making $S$ a \emph{graded {\bf M3}-monoid}. In particular,
\begin{equation}
\label{eq:braided_monoid}
\begin{array}{lclc}
{ML0 :}\quad & {}^a1=1,\quad  {}^1u=u,\quad &{MR0:} \quad &1^u=1,\quad a^1=a,
\\
 {ML1:}\quad& {}^{(ab)}u={}^a{({}^bu)},\quad& {MR1:}\quad  & a^{(uv)}=(a^u)^v,
 \\
{ML2:}\quad & {}^a{(u.v)}=({}^au)({}^{a^u}v),\quad &{MR2:}\quad &
(a.b)^u=(a^{{}^bu})(b^u),\\
{M3:}\quad &{}^uvu^v=uv& &
\end{array}
\end{equation}
hold in $S$ for all $a, b, u, v \in S$. These actions define a  map \[r_S: S\times S
\longrightarrow S\times S, \quad  r_S(u, v) := ({}^uv, u^v)\]
which obeys the Yang-Baxter equation, so $(S, r_S)$ is \emph{a braided
monoid}.
In particular,
$(S, r_S)$
is a set-theoretic solution of YBE, and the associated
map $r_S$ restricts to $r$.  Here, $r_S$ is a bijective map \emph{iff} $r$ is a bijection.
\item $(S,r_S)$ is \emph{a graded braided
monoid},  that is the actions agree with the grading (by length) of $S$:
\begin{equation}
\label{eq:braided_monoid2}
|{}^au|= |u|= |u^a|,\quad  \forall   \; a,u \in S.
\end{equation}
\item $(S, r_S)$ is left (resp. right) non-degenerate  \emph{iff} $(X,r)$ is left (resp. right) non-degenerate.
\end{enumerate}
In part  (2), we use the grading $S = \bigsqcup_{d\in\N_0}  S_{d}$  in (\ref{eq:Sgraded}). If we write  $\cA= \cA(\k , X, r)$ for
the associated YB-algebra, this is isomorphic to the monoid algebra $\k S$  and inherits the natural grading
$\cA = \k \oplus \cA_1 \bigoplus \cA_2\oplus\cdots,$  where $\cA_d \simeq \k S_d$ as vector spaces.
Each of the graded components $S_d, \; d \geq 1$, is $r_S$-invariant and we can consider
the restriction
$r_d = (r_S)_{|S_d\times S_d}$, where
$r_d: S_d\times S_d\longrightarrow S_d\times S_d$.

\begin{cor}
\label{dVeron_monoid}
Let $(X,r)$ be a braided set. Then for every positive integer $d \geq 1$, $(S_d, r_d)$ is a braided
set. Moreover, if $(X,r)$ is of finite order $n$, then $(S_d, r_d)$ is of order
\begin{equation}
\label{eq:fixedwords2}
   |S_d|  = |\cN_d| = \dim \cA_d.
\end{equation}
\end{cor}
\begin{dfn} \cite{GI_Veronese}
\label{def:VeroneseSol}
We call  $(S_d,r_d)$ \emph{the monomial $d$-Veronese solution associated with
$(X,r)$}.
\end{dfn}
The monomial $d$-Veronese solution $(S_d, r_d)$ depends only on the map $r$
and on the integer $d$, being invariant with respect to the enumeration of $X$. Although $(S_d, r_d)$ is intimately connected
with the $d$-Veronese subalgebra
of $\cA(\k,X,r)$ and its quadratic relations, it is not yet convenient for an
explicit description of those relations. We turn to this next.

\subsection{Normalized braided monoid and normalized $d$-Veronese solutions}

We show that the solution $(S_d, r_d)$  induces in a natural way an isomorphic solution
$(\cN_{d},
\rho_d)$. The fact that $\cN_{d}$ is ordered lexicographically makes this
solution convenient for our description of the  relations of the $d$-Veronese subalgebra. The set $\cN_{d}$, as a subset of the
set of normal monomials $\cN$, will depend
on the initial enumeration of $X$.

\begin{rmk}
    \label{rmk:the actions}
Note that given the monomials $a = a_1a_2 \cdots a_p \in X^p$ and $b=
b_1b_2\cdots b_q \in X^q$, we can find effectively the monomials  ${}^{a}{b}\in
X^q$ and $a^b \in X^p$. Indeed, as in \cite{GIM08}, we use the conditions
(\ref{eq:braided_monoid}) to extend the left and the right actions inductively:

\begin{equation}
\label{eq:effective_actions}
\begin{array}{lll}
{}^c{(b_1b_2\cdots b_q)}=({}^cb_1)({}^{c^{b_1}}b_2) \cdots ({}^{(c^{({b_1}\cdots
b_{q-1})})}b_q)),\quad \forall\ c \in X,\\
&&\\
{}^{(a_1a_2 \cdots a_p)}b= {}^{a_1}{({}^{(a_2 \cdots a_p)}b)}.
\end{array}
\end{equation}
We proceed similarly with the right action.
\end{rmk}

\begin{lem} \cite[Lemma 4.7]{GI_Veronese}
   \label{lem:Daction}
We use notation as in Remark~\ref{orbitsinG}.
   Suppose $a, a_1\in X^p,  a_1 \in \Ocal_{\Dcal_p} (a)$, and $b, b_1 \in X^q,
   b_1 \in \Ocal_{\Dcal_q} (b)$.
\begin{enumerate}
\item The following are equalities of words in the free monoid $\asX$:
\begin{equation}
  \label{eq:Noractions}
  \begin{array}{ll}
  \Nor ({}^{a_1}{b_1}) = \Nor ({}^{a}{b}), & \Nor ({a_1}^{b_1}) = \Nor
  ({a}^{b}).\\
  \Nor ({}^{a}{b}) = \Nor ({}^{\Nor(a)}{\Nor(b)}),& \Nor(a^b)= \Nor
  ({\Nor(a)}^{\Nor(b)}).
  \end{array}
\end{equation}
In particular, the equalities $a=a_1$
and $b=b_1$ in $S$ imply that
${}^{a_1}{b_1} ={}^ab$ and $a_1^{b_1}= a^b$ in $S$.
 \item  The following are equalities in the monoid $S$:
 \begin{equation}
  \label{eq:M3Noractions}
  ab = {}^{a}{b}a^b =  \Nor ({}^{a}{b})\Nor ({a}^{b}).
  \end{equation}
\end{enumerate}
\end{lem}

\begin{dfn}
\label{def:rho}
Define  left  and right `actions' on $\cN$ by
\begin{equation}
  \label{eq:actions}
\la : \cN \times \cN \longrightarrow \cN,\quad  a \la  b := \Nor
({}^ab);\quad   \ra : \cN \times \cN \longrightarrow \cN,\quad   a \ra  b := \Nor(a^b),\end{equation}
  for all $a, b \in \cN$.  Using these, we define the map
\begin{equation}
  \label{eq:rho}
\rho: \cN \times \cN \longrightarrow \cN \times \cN,
 \quad  \rho(a,b) := (a \la  b, a \ra  b).
\end{equation}
and its restriction $\rho_d=\rho|_{\cN_d\times\cN_d}$ as a map $\rho_d: \cN_{d} \times \cN_{d} \longrightarrow \cN_{d} \times
\cN_{d}$.
\end{dfn}
It follows from Lemma~\ref{lem:Daction}  (1)  that the two actions in (\ref{eq:actions}) are
well defined.

\begin{dfn}
\label{def:normalizedSol}
We call $(\cN,\rho)$ the \emph{normalized braided monoid associated with} $(X,r)$ and $(\cN_d,\rho_d)$ the \emph{normalised
$d$-Veronese solution associated with $(X,r)$}.
\end{dfn}

\begin{pro}
\label{dVeron_monoid2}
We use notation and assumptions as above.
\begin{enumerate}
\item $(\cN_{d}, \rho_d)$ is a
solution of the YBE of order $|\cN_{d}|$;
\item $(\cN_{d}, \rho_d)$  and $(S_d, r_d)$ are
    isomorphic solutions;
    \item $(\cN, \rho)$ is a solution isomorphic to $(S,r_S)$.
    \end{enumerate}
\end{pro}
\begin{proof}
(1) By Corollary~\ref{dVeron_monoid}, $(S_d, r_d)$ is a braided set. Thus, by Remark~\ref{rmk:YBE1}, the left and right actions
associated
with $(S_d, r_d)$ satisfy conditions \textbf{l1}, \textbf{r1}, \textbf{lr3}.
 Consider the actions $\la$ and $\ra$ on  $\cN_{d}$ given in
 Definition~\ref{def:rho}. It follows from  (\ref{eq:actions}) and
 Lemma~\ref{lem:Daction} that these actions also satisfy \textbf{l1},
 \textbf{r1}, and \textbf{lr3}.
Therefore, by Remark~\ref{rmk:YBE1} again, $ \rho_d$ obeys YBE, so $(\cN_{d}, \rho_d)$ is a braided set. It is clear that
$|\cN_{d}|=|S_d|$.

(2) We shall prove that the map $\Nor : S_d \longrightarrow \cN_d,$  $u
\mapsto \Nor (u)$ is an isomorphism of solutions. It is clear that the map is
bijective.
We have to show that $\Nor$ is a homomorphism of solutions, that is
\begin{equation}
  \label{eq:isom}
(\Nor \times \Nor) \circ r_d =  \rho_d\circ (\Nor \times \Nor).
\end{equation}
Let $(u, v)\in S_d\times S_d$, then the equalities $u =\Nor(u)$ and $v =
\Nor(v)$ hold in $S_d$, so
\[
\Nor ({}^uv)= \Nor ({}^{\Nor (u)}{\Nor (v)}), \quad  \Nor (u^v) =
\Nor({\Nor(u)}^{\Nor (v)}).\]
Together with (\ref{eq:actions}), this implies
\[\begin{array}{ll}
(\Nor \times \Nor) \circ r_d (u,v) &=  \Nor \times \Nor ({}^uv, u^v) =
(\Nor({}^uv), \Nor(u^v))\\
                                                &=(\Nor(u)\la \Nor(v), \Nor(u)
                                                \ra \Nor(v)) = \rho_d( \Nor (u),
                                                \Nor (v)).
\end{array}
\]
(3) The proof that $\Nor : S\longrightarrow \cN,$ $u
\mapsto \Nor (u)$ is an isomorphism of solutions is entirely similar.
\end{proof}

\subsection{Formulae for $(\cN,\rho)$ and $(\cN_d, \rho_d)$ in the permutation idempotent case}
Here, we specialise to a permutation idempotent solution of order $n$. We want to give a more precise description of $(\cN,
\rho)$ and $(\cN_d, \rho_d)$ in this case. We use the general setting above as well as the description of $\cN$ and $\cN_d$ in
Corollary~\ref{rmk:normalforms} for permutation idempotent case.

\begin{pro}
\label{proprho} Let $(X,r_f)$ be a permutation idempotent solution with $X = \{x_1, \cdots, x_n\}$. The associated monoid $(\cN,
\bullet)$ is a graded braided monoid with a braiding operator
\begin{equation}
\label{eq:broperator}
\rho : \cN\times \cN \longrightarrow \cN\times \cN,\quad \rho (x_1^{d-1}x_p, x_1^{m-1}x_q)= (x_1^{m-1} f^d(x_q),
x_1^{d-1}x_q),\quad \forall\ d, m \geq 2.
\end{equation}
Moreover, $(\cN,  \rho)$ is a left nondegenerate solution and  $\rho^3=\rho$, but $\rho^2\neq  \rho$ if $n\ge 2$.
\end{pro}
\begin{proof}
We first establish some formulae for the permutation idempotent case:

(1) The condition \textbf{l1} and ${}^xy= f(y)$ for all $x,y \in X$ imply
 \begin{equation}
\label{eq:leftaction1}
{}^{y_1y_2\cdots y_{d}}{x_q} = f^d(x_q),\quad\forall\  d \geq 1, \; y_i\in X, \; 1\leq i \leq d,\  q \in \{1, \cdots, n\},
\end{equation}
from which it follows that the left action of $S=S(X,r_f)$ on itself is by automorphisms:
\[{}^a{(uv)}=({}^a{u})({}^{a^{u}}v) = ({}^a{u})({}^{a}v), \quad \forall \; a, u, v \in S.\]

(2) The equality  $x^y = y$, condition \textbf{r1} and induction imply that
\[x^{y_1y_2\cdots y_m}= y_m,\quad \forall\ m\ge 1,\  x, y_1, \cdots, y_m \in X. \]

(3) Now let $u=y_1y_2\cdots y_{m-1}x_q\in X^m, z_1\cdots z_d \in X^d$ and iterate \emph{MR2}  to obtain as equalities in $X^d$,
\[
\begin{array}{ll}
(z_1\cdots z_d)^{u} &= (z_1)^{{}^{(z_2\cdots z_d)}u}\cdots (z_{d-1})^{({}^{z_d}u)}  (z_d)^u = t_1 \cdots t_{d-1} (z_d)^u= t_1
\cdots t_{d-1} x_q
                  \end{array}\]
                  for some $t_i \in X, 1 \leq i \leq d-1$. Hence,  by (\ref{eq:Nor1}),
\begin{equation}
\label{eq:righttaction1}
\Nor((z_1\cdots z_d)^{y_1y_2\cdots y_{m-1}x_q}) = \Nor(t_1 \cdots t_{d-1} x_q)= (x_1)^{d-1}x_q.
\end{equation}
In particular,
\begin{equation}
\label{eq:righttaction2}
\Nor((x_1^{d-1}x_p)^{(x_1^{m-1}x_q)}) = x_1^{d-1} x_q,\quad \forall\ d, m \geq 2.
\end{equation}

Using these results, the map  $\rho: \cN\times \cN\longrightarrow  \cN\times \cN$ in  Definition~\ref{def:rho} is
\[ \rho (x_1^{d-1}x_p, x_1^{m-1}x_q)=\Nor({}^{x_1^{d-1}x_p}{(x_1^{m-1}x_q)}), \Nor((x_1^{d-1}x_p)^{(x_1^{m-1}x_q)}),\]
which comes out as stated, using part (1) for the first component and part (3) for the second component.  The equalities
(\ref{eq:broperator}) then imply straightforwardly that $\rho^3=\rho$, but $\rho^2\neq\rho$.
\end{proof}

\begin{cor}
\label{cor:d-Ver_normalizedsolution} In the setting of Proposition~\ref{proprho}, let $d \geq 2$ be an integer and
$(\cN_d, \rho_d)$ the normalized $d$-Veronese solution associated to $(X,r_f)$.
Then $(\cN_d, \rho_d)$ is again a permutation idempotent solution of order $n$. Moreover, if we enumerate lexicographically,
\[\cN_d = \{w_1 = x_1^d ,\ w_2= x_1^{d-1}x_2,\ \cdots,\ w_n= x_1^{d-1}x_n\}\]
then
\begin{equation}
\label{eq:Fd}
 \rho_d (w_p, w_q) =  (F(w_q), w_q);\quad
F\in\Sym(\cN_d),\quad F (w_q)= x_1^{d-1}f^d(x_q),\quad \forall\ 1 \leq p, q\leq n.
\end{equation}
\end{cor}
\begin{proof}
For each $d \geq 2$, the braiding operator $\rho$ in Proposition~\ref{proprho} restricts to  a map   $\rho_d : \cN_d\times \cN_d
\longrightarrow \cN_d\times \cN_d$ given by
\begin{equation}
\label{eq:righttaction2a}
\rho_d (x_1^{d-1}x_p, x_1^{d-1}x_q) =  (x_1^{d-1} f^d(x_q), x_1^{d-1}x_q),\quad \forall\ d \geq 2
\end{equation}
which is of the form stated. Here, $f^d(x_q)=x_{q'}$ for some $q'$ and in this case $F(w_q)=w_{q'}$  so that $F\in \Sym(\cN_d)$.
\end{proof}
We see that the $d$-Veronese solution $(\cN_d, \rho_d)$ is in the class $P_n$ of all permutation idempotent solutions of order
$n$, namely given by the iterated permutation $f^d$ if we use the enumerations given.

\subsection{Veronese subalgebras and morphisms for permutation idempotent solutions}
\label{sec:dVeronese}

We are now ready to find the $d$-Veronese subalgebras of $\cA(\k,X,r_f)$ as isomorphic to $\cA(\k,\cN_d,\rho_d)$. We first recall
some basic definitions and facts about Veronese subalgebras of
general graded algebras, as in the text \cite[Sec.~3.2]{PoPo}.
\begin{dfn}
Let $A= \bigoplus_{m\in\N_0}A_{m}$ be a graded $\k$-algebra. For any integer $d\geq
2$, the \emph{$d$-Veronese subalgebra} of $A$ is the graded algebra
\begin{equation}
\label{eq:A^d}
A^{(d)}=\bigoplus_{m\in\N_0} A_{md}.
\end{equation}
\end{dfn}

By definition, the algebra $A^{(d)}$ is a subalgebra of $A$. However, the
embedding is not a graded algebra morphism.
The Hilbert function $h_{A^{(d)}}$ of $A^{(d)}$ satisfies
    \[h_{A^{(d)}}(m)=\dim(A^{(d)})_m=\dim(A_{md})=h_A(md).\]

It follows from \cite[Prop.~2.2, Chap.~3]{PoPo}
that if $A$ is a one-generated quadratic Koszul algebra then its Veronese
subalgebras are also one-generated quadratic and Koszul.
Moreover, \cite[Prop.~4.3, Chap.~4]{PoPo} implies  that if $x_1, \cdots, x_n$ is a set of PBW-generators of a
PBW algebra $A$, then the elements of its
PBW-basis of degree $d$, taken in lexicographical order, are PBW-generators of the Veronese subalgebra $A^{(d)}$.

In the remainder of this section, we let $\cA = \cA(\k , X,r_f)$  be the Yang-Baxter algebra of a permutation idempotent solution
$(X,r_f)$ or order $n$, where $X = \{x_1, \cdots, x_n\}$ and $f \in \Sym (X).$

\begin{cor}
\label{cor:dVeron}
Given $(X,r_f)$  and $d \geq 2$ an integer, the  $d$-Veronese subalgebra $\cA^{(d)}$  is a
PBW algebra with PBW-generators the set
 \begin{equation}
 \label{eq:Nd}
 \cN_d = \{w_1 = x_1^d < w_j = x_1^{d-1}x_2 <\cdots < w_n = x_1^{d-1}x_\}
 \end{equation}
of normal monomials of length $d$ ordered lexicographically.
\end{cor}
\begin{proof}
It follows from \cite[Prop.~4.3, Chap.~4]{PoPo} that if $x_1, \cdots, x_n$ is a set of PBW-generators of a PBW algebra $A$  then
the elements of its PBW $\k$-basis of degree $d$, taken in lexicographical order, are PBW-generators of the Veronese subalgebra
$A^{(d)}$.
By Theorem~\ref{thm:YBalgPermSol}, our algebra $\cA$ is PBW, which implies straightforwardly the result.
\end{proof}

\begin{thm}
\label{thm:veronesealg}
Given $(X,r_f)$ and $d \geq 2$ an integer, let $\cN_d$ be the set of normal monomials of length $d$ ordered lexicographically.
\begin{enumerate}
\item
The $d$-Veronese subalgebra $\cA^{(d)}$ of $\cA$ is a PBW algebra with a set of one-generators $\cN_d$ and a standard finite
presentation
\begin{equation}
\label{eq:rels_ver}
\begin{array}{c}
\cA^{(d)} = \k \langle w_1, \cdots w_n\rangle  /(\mathcal{R}_d);\quad
\mathcal{R}_d = \{w_jw_p- w_1w_p\mid 2 \leq j \leq n, 1 \leq p\leq n\},
\end{array}
\end{equation}
where $\mathcal{R}_d$  consists of $n(n-1)$ binomial relations and forms a Gr\"{o}bner basis of the two sided ideal $I=
(\mathcal{R}_d)$ in $ \k \langle w_1, \cdots w_n\rangle$.
\item
The algebra $\cA$ and its Veronese subalgebra $\cA^{(d)}$ are isomorphic.
\end{enumerate}
\end{thm}
\begin{proof}
By Convention
\ref{rmk:conventionpreliminary1},  we identify the algebra  $\cA$ with
$(\k \cN, \bullet ).$ By (\ref{eq:A^d}),
\[
\cA^{(d)} = \bigoplus_{m\in\N_0}  \cA_{md} \cong \bigoplus_{m\in\N_0}
\k \cN_{md}.
\]
So $\cA^{(d)}_1 =\k \cN_{d}$ and the monomials
$w \in \cN_d$ of length $d$ are degree one generators of $\cA^{(d)}$. We have
 \[
\dim \cA_d = |\cN_d |=n,\quad
\dim((\cA^{(d)})_2) = \dim(\cA_{2d}) = \dim (\k \cN_{2d}) = n.
\]

We want to find  a finite  presentation of $\cA^{(d)}$ in terms of generators and relations
\[\cA^{(d)} = \k \langle w_1, \cdots, w_n\rangle/ I,\]
where the two-sided (graded) ideal $I$ is generated by linearly independent homogeneous relations $R$ of degree $2$ in the
variables $w_i$, with
$I_2 = \Span_{\k } R$. Moreover, we have
\[\k \langle w_1, \cdots, w_n\rangle_2 = I_2 \oplus \k \cN_{2d}\]
and hence,
\begin{equation}\label{eq:I2}\dim I_2 + \dim \k \cN_{2d} = \dim (\k \langle w_1, \cdots, w_n\rangle_2),\quad \dim I_2 = n^2 - n =
n(n-1). \end{equation}

First, we prove that each quadratic polynomial in $\mathcal{R}_d$ as defined in (\ref{eq:rels_ver}) is a relation of $\cA^{d}$.
Note that each equality in $(\k \cN, \bullet )$ is also an equality in $\cA$.
We shall use the normalized $d$-Veronese solution $(\cN_{d},  \rho_d)$, which we know from
Corollary~\ref{cor:d-Ver_normalizedsolution} is a permutation idempotent solution
$(\cN_d, r_F)$ where  $F (w_q)= x_1^{d-1}f^d(x_q)$. This implies that
\[
w_p.w_q = F(w_q).w_q,\quad\forall\  1 \leq p, q\leq n
\]
as equalities in $(\cN, \bullet)$. In particular,
\[w_p.w_q = F(w_q).w_q= w_1.w_q,\quad\forall\  1 \leq p, q\leq n  \]
 are equalities in $(\cN, \bullet)$, which implies that each of the quadratic polynomials
\[
w_pw_q  - w_1w_q \in \k \langle w_1, \cdots, w_n\rangle,\quad  \; 2 \leq p\leq n, 1 \leq q \leq n,
\]
is identically zero in $\cA$,  and hence in $\cA^{(d)}$.
But these are exactly the elements of
$\mathcal{R}_d$. Hence,
 $\mathcal{R}_d\subseteq I_2$, the degree 2 part of the ideal of relations for the $d$-Veronese subalgebra $\cA^{(d)} = \k
 \langle w_1, \cdots, w_n\rangle /I$.
Note that the relations in $\mathcal{R}_d$ are linearly independent, since these are $n(n-1)$ relations whose leading monomials
$w_pw_q,\  2 \leq p \leq n, 1\leq q \leq n$, are pairwise distinct.
(It is well known that the set of all words in the alphabet $w_1, \cdots, w_n$ forms a basis of the free associative algebra $\k
\langle w_1, \cdots, w_n\rangle$, so any finite set of pairwise distinct words in the $w_i$'s is linearly independent).
Therefore, $\dim \k \mathcal{R}_d = n(n-1)= \dim I_2$ as required in (\ref{eq:I2}), which implies that
\[ I_2 = \k \mathcal{R}_d.\]
By Corollary~\ref{cor:dVeron}, the ideal $I$ of relations of $\cA^{(d)}$ is generated by quadratic polynomials, $I = (I_2)$.
It follows that $I =(\mathcal{R}_d)$, so $\cA^{(d)}$ is a quadratic algebra with an explicit presentation (\ref{eq:rels_ver}), as
desired.

Finally, the set $\mathcal{R}_d$ is the reduced Gr\"{o}bner basis of the two sided ideal $I= (\mathcal{R}_d)$ of the free
associative algebra $\k \langle w_1, \cdots, w_n\rangle$. The proof is analogous to the proof of Theorem~\ref{thm:YBalgPermSol}.
Therefore, $\cA^{(d)}$ is a PBW algebra with PBW-generators $w_1, \cdots, w_n$.
\end{proof}

Observe that (\ref{eq:rels_ver}) is also a presentation of the Yang-Baxter algebra of the finite permutation solution
$ (\cN_d, \rho_d)= (\cN_d, r_F)$ of order $n$. We have also seen that the Yang-Baxter algebras for permutation idempotent
solutions of a given order are isomorphic.

\begin{cor} Given  $(X,r_f)$ as above and for each $d\geq 2$,  the $d$ -Veronese subalgebra $\cA^{(d)}$  is isomorphic to the
Yang-Baxter algebra of the normalized $d$-Veronese solution $(\cN_d, \rho_d)$, which is also a permutation idempotent solution of
order $n$.
\end{cor}

\begin{cor}
\label{cor:veronese_map}
Let
$(Y, r_\mathfrak{F})$ be the permutation idempotent solution on a set $ Y = \{y_1 , \cdots y_n\}$, where $\mathfrak{F}$ is the
permutation of $Y$ given by
$\mathfrak{F} (y_q)= y_j$  iff $F(w_q) = x_1^{d-1}f^d(x_q)= w_j$ as in (\ref{eq:Fd}), and let $\mathfrak{A}$ be its Yang-Baxter
algebra,
\[
\; \mathfrak{A} = \k \langle y_1, \cdots y_n\rangle  /(\Re_d);\quad  \Re_d = \{y_py_q -y_1y_q\mid 2\leq p \leq n, 1 \leq q \leq
n\}.
\]
The  assignment
\[
y_i \mapsto w_i,\quad 1\le i\le n
\]
extends to an injective homomorphism  of graded algebras
$v_{n, d}: \mathfrak{A} \longrightarrow \cA,$ called the $n,d$-\emph{Veronese map}.
The image of $v_{n,d}$ is the $d$-Veronese subalgebra $\cA^{(d)}$.
\end{cor}
\begin{rmk}
\label{rmk:strategy}
This is in contrast to the involutive case in \cite{GI_Veronese}, where the kernel of the Veronese map $v_{n,d}$ is large. The
underlying general strategy, however, was similar, namely as follows. Given a solution $(X,r)$ of the YBE (and its YB-algebra
$\cA$), first determine a normalized $d$-Veronese solution  $(\cN_d, \rho_d)$ on a set $\cN_d = \{w_i\}$ (in our case $\cN_d$ has
exceptionally the same cardinality $n$, but in general $|\cN_d| = \dim \cA_d = N \ge n$) and use this to find an explicit
presentation of the $d$-Veronese subalgebras $\cA^{(d)}$ with generators $\{w_i\}$ and linearly independent quadratic relations
$\Re_d$. Usually, the set $\Re_d$ splits into two set of relations: (a) relations coming from the YB-algebra $\mathfrak{A}= \cA(
\k , \cN_d, \rho_d)$ (relations  $\Re_a$, say) and (b) an additional set of relations $\Re_b$, say, which do not come from the
solution $(\cN_d, \rho_d)$  but are identically zero in $\cA^{(d)}$. Here, $\Re_d = \Re_a\bigcup \Re_b$ is a disjoint union. The
total number of linearly independent relations $\Re_d$ must agree with $\dim ((\cA^{(d)})_2)$. In our present case, the relations
$\Re_b$ were missing. Next, to define the Veronese map $v_{n,d}$,  take an abstract solution
$(Z, \mathfrak{r})$,   with elements $\{z_i\}$,  which is isomorphic to the normalized $d$-Veronese solution $(\cN_d, \rho_d)$,
and consider its YB-algebra $B= \cA(k, Z, \mathfrak{r})$. (In general,  $Z$ has cardinality $N = |\cN_d|\geq n$,  but  in our
case  $|Z|= n$).
Finally, define an algebra homomorphism $v_{n,d} : B  \to  \cA$  extending the assignment $z_i\mapsto w_i$ whose image is the
$d$-Veronese subalgebra $\cA^{(d)}$ and describe the kernel  $K  = \ker v_{n,d}$,  so that  $\cA ^{(d)}\cong B/K$. In the general
case, the kernel $K$ is generated by linearly independent quadratic polynomials which are not identically zero in $B$ and are
mapped onto the relations of the second set $\Re_b$ of relations of $\cA^{(d)}$. In our particular case, the kernel was zero and
hence the Veronese map
$v_{n,d}: B \to  \cA ^{(d)}$ is an isomorphism.
\end{rmk}

\section{Segre Products and morphisms for YB-algebras of permutation idempotent solutions}

Here we will obtain Segre products and Segre morphisms for the Yang-Baxter algebras associated to permutation idempotent
solutions. This time we follow the general strategy of  \cite{GI23}. In particular, we always involve the Cartesian product of
solutions $(X\times Y, \rho_{X\times Y})\simeq (X\circ Y, r_{X\circ Y})$ as there, but due to the different nature of permutation
idempotent solutions, our results in Theorem~\ref{thm:rel_segre} and Corollary~\ref{thm:segremap} will be very different from the
results in  \cite[Thm.~3.10, Thm.~4.5]{GI23}. In our case, for every $m, n \geq 2$, the $m,n$-Segre map $s_{m, n}$ has a trivial
kernel and gives an isomorphism of algebras.

\subsection{Segre products of quadratic algebras}

We first  recall the Segre product of graded algebras as in  \cite[Sec.~3.2]{PoPo}. The theory goes back to  \cite{Froberg}
Fr\"{o}berg and Backelin,  who made a systematic account for Koszul algebras and showed  that their properties are preserved
under various constructions such as tensor products, Segre products, Veronese subalgebras.
 An interested reader can also find results on the Segre product of specific Artin-Schelter regular algebras in \cite{kristel},
 and on twisted Segre products of Noetherian Koszul Artin-Schelter regular algebras in \cite{twistedSegre}.

\begin{dfn}
\label{dfn:segre}
    Let \[
    A=  \k  \oplus A_1\oplus A_2   \oplus \cdots,\quad   B=   \k  \oplus B_1\oplus B_2   \oplus \cdots\]
    be $\N_0$-graded algebras over a field $\k $,  where $\k  = A_0 = B_0$.
The \emph{Segre product} of $A$ and $B$ is the $\N_0$-graded algebra
    \[A\circ B:=\bigoplus_{i \geq 0}(A\circ B)_i;\quad  (A\circ B)_i =  A_i\otimes_{\k } B_i. \]
\end{dfn}

The Segre product $A\circ B$ is a subalgebra of the tensor product algebra $A\otimes B$, but note that this  embedding is not a
graded algebra morphism as it doubles the grading. If $A$ and $B$ are locally finite then the Hilbert function and Hilbert series
of $A\circ B$ obviously satisfy
  \[
    h_{A\circ B}(m)=\dim(A\circ B)_m=\dim(A_m\otimes B_m)=\dim(A_m)\dim(B_m)=h_A(m) h_B(m),
 \]
 \[
H_A(t)=  \Sigma _{n \geq 0} (\dim A_n) t^n,  \; \; H_B(t)=  \Sigma _{n \geq 0} (\dim B_n) t^n, \;\;  H_{A\circ B}(t)=  \Sigma _{n
\geq 0} (\dim A_n)(\dim B_n) t^n.
\]

The Segre product $A \circ B$ inherits various properties from the two algebras $A$ and $B$. In particular, if both algebras are
one-generated, quadratic and Koszul, it follows from \cite[Prop.~2.1, Chap~3.2]{PoPo}  that  the algebra
$A \circ B$ is also one-generated, quadratic and Koszul. The following remark gives more concrete information about the space of
quadratic relations  of $A\circ B$, see \cite{kristel}.

\begin{rmk}
\label{rmk:relAoB1} \cite{kristel}
Suppose that $A$ and $B$ are quadratic algebras generated in degree one by $A_1$ and $B_1$, respectively, \[
A= T(A_1)/ (\Re_A),\quad \Re_A \subset A_1\otimes A_1;\quad
B= T(B_1)/ (\Re_B), \quad  \Re_B \subset B_1\otimes B_1,
\]
where $T(-)$ is the tensor algebra and $(\Re_A), (\Re_B)$ are the ideals of relations of $A$ and $B$. Then  $A\circ B$ is also a
quadratic algebra generated in degree one by $A_1\otimes B_1$ and presented as
\begin{equation}
\label{eq:segrelations1}
A \circ B = T(A_1\otimes B_1)/ (\sigma^{23}( \Re_A \otimes B_1\otimes B_1 +A_1\otimes A_1\otimes \Re_B)),
\end{equation}
where $\sigma^{23}$ is the flip map in the 2nd and 3rd tensor factors (and the identity on the other tensor factors).
\end{rmk}

We also note a straightforward consequence of \cite[Prop.~2.1, Chap.~3]{PoPo}.
\begin{rmk}
\label{cor:SegreProductProperties}
Let $(X, r_X)$ and $(Y, r_Y)$ be finite braided sets and let $A= \cA(\k , X, r_X)$ and $B=\cA(\k , Y, r_Y)$ be their Yang-Baxter
algebras. Then the Segre product $A \circ B$
is a one-generated quadratic algebra.
\end{rmk}

\subsection{Segre products in the case of permutation idempotent solutions}

We first recall the following definition.

\begin{dfn} \cite{GI23}
\label{dfn:cartesianproductof solutions}
    Let $(X,r_X)$ and $(Y, r_Y)$ be disjoint braided sets (we do not assume involutiveness, nor nondegeneracy). Their
    \emph{Cartesian product $(X\times Y, \rho_{X\times Y})$} is a braided set
with $\rho=\rho_{X\times Y}$ given by
\[\rho: (X\times Y)\times  (X\times Y) \longrightarrow (X\times Y)\times  (X\times Y),\quad \rho= \sigma_{23}\circ (r_X\times
r_Y)\circ \sigma_{23},  \]
where $\sigma_{23}$ is the flip of the 2nd and 3rd components. Explicitly, if $|X|=m$ and $|Y|=n$,
\begin{equation}
\label{eq:cartproduct}
\rho((x_j, y_b),(x_i, y_a)) := (({}^{x_j}{x_i},  {}^{y_b}{y_a}), (x_j^{x_i}, y_b^{y_a})),
\end{equation}
for all $i, j \in \{1, \cdots, m\}$ and all  $a, b \in \{1, \cdots, n\}$.
The Cartesian product $(X\times Y, \rho_{X\times Y})$ in this case is a braided set of order $mn$.
\end{dfn}

Henceforth and until the end of the section,  we assume that $(X,r_f)$ and $(Y, r_{\varphi})$ are permutation idempotent
solutions of the YBE,  where $f \in \Sym(X), \varphi \in \Sym (Y)$.

\begin{lem}\label{lemseg} Given $(X,r_f)$ and $(Y, r_{\varphi})$ permutation idempotent solutions of the YBE,  the Cartesian
product
$(X\times Y, \rho_{X\times Y})$ is a permutation idempotent solutions of the YBE, namely
\[(X\times Y, \rho_{X\times Y}) = (X\times Y, r_{\Phi}), \quad \Phi= f\times \varphi \in \Sym (X\times Y).  \]
Explicitly, if $|X|=m$ and $|Y|=n$,
\begin{equation}
\label{eq:cartprodPermSol}
\rho((x_j, y_b),(x_i, y_a)) = (\Phi(x_i,y_a), (x_i, y_a))),
\end{equation}
for all $i, j \in \{1, \cdots, m\}$ and all  $a, b \in \{1, \cdots, n\}$.
\end{lem}
\begin{proof}
This is immediate from Definition~\ref{dfn:cartesianproductof solutions}. In the finite case,
\begin{equation}
\label{eq:cartprodPermSol1}
\rho((x_j, y_b),(x_i, y_a)) = ((f(x_i),  \varphi(y_a), (x_i, y_a))),
\end{equation}
for all $i, j \in \{1, \cdots, m\}$ and all  $a, b \in \{1, \cdots, n\}$.
\end{proof}

We henceforth only work with the finite case and fix enumerations
\begin{equation}\label{XYenum} X = \{x_1, \cdots, x_m\}, \quad  \quad Y = \{y_1, \cdots, y_n\}\end{equation}
as well as degree-lexicographic orders on the free monoids $\langle X\rangle$, and  $\langle Y\rangle$ extending these
enumerations.
Let
 $A =\cA(\k ,  X, r_f)$ and  $B =\cA(\k ,  Y, r_\varphi)$  be the corresponding Yang-Baxter algebras.
 By Theorem~\ref{thm:YBalgPermSol}, each of them has two equivalent presentations which will be used in the sequel, namely
 \begin{equation}
\label{eq:relations1_A, B}
\begin{array}{c}
A = \k \langle X  \rangle /(\Re_A);\quad
\Re_A = \{x_jx_p- f(x_p)x_p\mid 1 \leq j,p \leq m\},\\
B = \k \langle Y  \rangle /(\Re_B);\quad
\Re_B = \{y_by_a- \varphi(y_a)y_a\mid 1 \leq a,b \leq n\},\\
\end{array}
\end{equation}
and the standard finite presentations
\begin{equation}
\label{eq:relations_A, B}
\begin{array}{c}
A = \k \langle X  \rangle /(\Re_1);\quad
\Re_1 = \{x_jx_p- x_1x_p\mid 2 \leq j \leq m, 1 \leq p\leq m\},\\
B= \k \langle Y  \rangle /(\Re_2);\quad
\Re_2 = \{y_by_a- y_1y_a \mid 2 \leq b \leq n, 1 \leq a\leq n\}.
\end{array}
\end{equation}
Here, $\Re_1$  is a set of $m(m-1)$ binomial relations and $\Re_2$ of $n(n-1)$ binomial relations. One has
\begin{equation}
\label{eq:dimensions}
\dim A_s= m, \quad  \dim B_s= n, \quad \dim (A\circ B)_s = mn,\quad\forall s \geq 1.
\end{equation}

\begin{convention}
\label{notation:x_circ_y}
To simplify notation when we work with elements of the Segre product $A\circ B$, we will write $x\circ y$ instead of $x\otimes y$
whenever $x\in X, y\in Y$,
or $u\circ v$ instead of $u\otimes v$  whenever $u \in A_d,$ $v \in B_d,\  d \geq 2.$
\end{convention}
\begin{lem}
\label{pronotation}
Let $(X, r_f)$ and $(Y, r_\varphi)$ be as in Lemma~\ref{lemseg}, let $A\circ B$ be the Segre product of the YB algebras $A=
\cA(\k , X, r_f)$ and $B=\cA(\k, Y, r_{\varphi})$,
and let
\[X\circ Y:=\{ x_i\circ y_a \mid  1 \leq i \leq m, \; 1\leq a \leq n\}.\]
There is a natural structure of a permutation idempotent solution
\begin{equation}
\label{eq:psi}
\begin{array}{ll}
(X\circ Y, r_{X\circ Y})= &(X\circ Y, r_{\Psi});\quad  \Psi\in \Sym (X\circ Y),\quad \Psi (x\otimes y)= f(x)\otimes  \varphi(y)
\end{array}
\end{equation}
of order $mn$ and isomorphic to the Cartesian product of the original solutions. Explicitly,
\begin{equation}
\label{eq:def_r}
r_{X \circ Y}(x_j \circ y_b, x_i\circ y_a) := (f(x_i)\circ \varphi({y_a}),     (x_i\circ y_a)),
\end{equation}
for all $1 \leq i, j\leq m$ and all $1 \leq a, b \leq n.$
\end{lem}
\begin{proof} This is again immediate.
The set $X\circ Y$ consists of $mn$ distinct elements and is a basis of $(A\circ B)_1 = A_1\otimes B_1$. Moreover, the map $r:
(X\circ Y )\times (X\circ Y )\longrightarrow (X\circ Y )\times (X\circ Y )$ defined via (\ref{eq:def_r}) is well defined. The
isomorphism with the Cartesian product is straightforwardly via the bijective map $F: X\circ Y\rightarrow X \times Y$ given by
$F(x\circ y) = (x, y)$.
\end{proof}
We shall therefore identify the solutions $(X\circ Y, r_{\Psi})$ and $(X\times Y, \rho_{\Phi})$ and refer to $(X\circ Y,
r_{X\circ Y})$ as `the Cartesian product', in the present case of solutions $(X, r_f)$ and $(Y, r_\varphi)$.

  \begin{lem}
\label{lem:segreproductrel} In the context of Lemma~\ref{pronotation}
and
for each $1 \leq i, j\leq m$ and $1 \leq a,b \leq n$, one has
\[
f_{jb,ia}:= (x_j\circ y_b)(x_i\circ y_a) - (f(x_i)\circ \varphi(y_a))  (x_i\circ y_a) \in (\Re(A\circ B)).
\]
\end{lem}
\begin{proof} By (\ref{eq:relations1_A, B}),
$x_jx_i- f(x_i)x_i\in \Re_A$ and $y_by_a- \varphi(y_a)y_a \in \Re_B.$
Then, by Remark~\ref{rmk:relAoB1},
\[
\begin{array}{ll}
\psi_1 &=\sigma_{23}((x_jx_i-  f(x_i)x_i)\circ(y_by_a))\\
            &=(x_j\circ y_b)(x_i\circ y_a)- (f(x_i)\circ y_b)(x_i\circ y_a)
\in (\Re(A\circ B)),\\
\psi_2 &= \sigma_{23}(f(x_i)x_i\circ(y_by_a - \varphi(y_a)y_a))\\
       &=(f(x_i)\circ y_b)(x_i\circ y_a) -
            (f(x_i) \varphi(y_a))(x_i\circ y_a ) \in (\Re(A\circ B)).
\end{array}
\]
The elements $\psi_1$ and $\psi_2$ are in the ideal of relations  $(\Re(A\circ B))$, so the sum
\[
\begin{array}{ll}
\psi_1 + \psi_2 &=(x_j\circ y_b)(x_i\circ y_a)- (f(x_i)\circ y_b)(x_i\circ y_a)+(f(x_i)\circ y_b)(x_i\circ y_a) -
                            (f(x_i) \varphi(y_a))\circ (x_i\circ y_a )\\
                 &= (x_j\circ y_b)(x_i\circ y_a)- (f(x_i) \varphi(y_a))\circ (x_i\circ y_a )
                  = f_{jb,ia}
\end{array}
\]
is also in $(\Re(A\circ B))$. \end{proof}

\begin{thm}
\label{thm:rel_segre}
Let $(X, r_f)$ and $(Y, r_{}\varphi)$ be finite permutation idempotent solutions,  $A\circ B$ the Segre product of the YB
algebras $A= \cA(\k , X, r_f)$ and $B=\cA(\k , Y, r_{\varphi})$ and $(X\circ Y, r_{X\circ Y})$ the Cartesian product of solutions
from Lemma~\ref{pronotation}.
\begin{enumerate}
\item
The algebra $A\circ B$ is a PBW algebra with a set of $mn$ PBW one-generators $\mathcal{W}=X\circ Y$
ordered lexicographically,
\begin{equation}
\label{eq:W}
\mathcal{W} = \{w_{11}= x_1\circ y_1< w_{12}= x_1\circ y_2 <\cdots <w_{1n}= x_1\circ y_n< \cdots < w_{mn}= x_m\circ y_n\},
\end{equation}
and a finite standard presentation $A\circ B = \k \langle \mathcal{W}\rangle/(\Re)$, where
\begin{equation}
\label{eq:SegrePBW}
\quad \Re = \{F_{jb, ia}=w_{jb}w_{ia}- w_{11}w_{ia}\mid (j, b) \neq (1, 1),\  1 \leq i,j \leq m,\  1 \leq a,b \leq n\}
\end{equation}
is a set of $mn(mn-1)$ quadratic relations. The leading monomial of  each $F_{jb, ia}$ is $\LM (F_{jb, ia})= w_{jb}w_{ia}$.
Moreover,  $\Re$ is a reduced Gr\"{o}bner basis of the ideal $I = (\Re)$ of the free associative algebra $\k \langle
\mathcal{W}\rangle$.
\item
The ideal $I$ has a second set of generators, $I =(\Re(A\circ B))$, where
\begin{equation}
\label{eq:Segre1}
\begin{array}{ll}
\Re(A\circ B)= \{
                f_{jb, ia}=w_{jb}w_{ia}- \Psi(w_{ia})w_{ia} \mid 1 \leq i,j \leq m, \; 1 \leq a,b \leq n\};\\
                \Psi \in \Sym(\mathcal{W}),\quad\Psi(x_i\circ y_a)=f(x_i)\circ \varphi(y_a),\quad \forall i \in \{1, 2, \cdots,
                m\},\  a \in \{1, 2, \cdots, n\}.
\end{array}
\end{equation}
\item The Segre product $A \circ B$ is isomorphic to the Yang-Baxter algebra $\cA (\k , X\circ Y, r_{X\circ Y})$ of the
Cartesian product  $(X\circ Y, r_{X\circ Y})$.
\end{enumerate}
\end{thm}
\begin{proof}
We know from Remark~\ref{cor:SegreProductProperties} that $A\circ B$ is a one-generated quadratic algebra, and $\mathcal{W}$ is
its  set of one-generators since $\Span_\k  (\mathcal{W})= (A\circ B)_1$. Therefore, $A\circ B = \k \langle
\mathcal{W}\rangle/I$, where the ideal of relations $I$  is generated by homogeneous polynomials of degree two.
We shall prove the equality of ideals
\[I = (\Re(A\circ B))= (\Re).\]
There is an equality of vector spaces
\[
(\k \langle \mathcal{W}\rangle)_2 = I_2 \oplus (A\circ B)_2.
\]
Moreover,
\begin{equation}
\label{eq:I}
I = (I_2),\quad  \dim I_2 = (mn)^2- \dim (A\circ B)_2 = (nm)^2-nm= nm(nm-1).
\end{equation}
By Lemma~\ref{lem:segreproductrel}, one has $\Re(A\circ B) \subseteq I_2$.

Observe that each of the polynomials $F_{jb, ia}=w_{jb}w_{ia}- w_{11}w_{ia}\in \Re$  is also in the ideal $(\Re(A\circ B))$.
More precisely,
\begin{equation}
\label{eq:Fjbia}
F_{jb, ia}= f_{jb, ia} - f_{11, ia}\in (\Re(A\circ B)).
\end{equation}
Indeed, by Lemma~\ref{lem:segreproductrel},  $f_{jb,ia}, f_{11,ia}$ are in the ideal $ (\Re(A\circ B)$, and
\[f_{jb, ia} - f_{11, ia} = (w_{jb}w_{ia}- \Psi(w_{ia})w_{ia})- (w_{11}w_{ia}- \Psi(w_{ia})w_{ia}) =w_{jb}w_{ia} -w_{11}w_{ia}
=F_{jb, ia}.\]
 This implies inclusions of vector spaces
 \begin{equation}
 \label{eq:subspaces}
 I_2 \supseteq \Span_\k  \Re(A\circ B) \supseteq \Span_\k  \Re.
 \end{equation}
Note that the set $\Re$ consists of $mn(mn-1)$ linearly independent binomials.
Indeed, the polynomials $F_{jb, ia}$ have pairwise distinct leading monomials
\[\LM(F_{jb, ia}) = w_{jb}w_{ia},\quad\forall\ (j,b) \neq (1,1),\ 1 \leq i,j \leq m,\   1 \leq a, b \leq n,\]
and therefore the set of all $F_{jb, ia}$  is linearly independent. But $\{F_{jb, ia} \mid 2 \leq j \leq m, 1 \leq i \leq m, 1
\leq a,b\leq n\} = \Re$.
It follows that
\[\dim\Span _\k (\Re)= mn(mn-1) = \dim I_2\]
and therefore (\ref{eq:subspaces}) consists of equalities,
\begin{equation}
 \label{eq:subspaces1}
 I_2 = \Span_\k  \Re(A\circ B) = \Span_\k  \Re.
 \end{equation}
However, the ideal $I$ is generated by $I_2$, $I = (I_2)$, so there are equalities of ideals
\[
I = (\Re(A\circ B)) = (\Re).
\]
This proves that each of the sets $\Re$ and $\Re(A\circ B)$ determines the ideal of relations of the Segre product $A \circ B$.
To verify that $\Re$ is a Gr\"{o}bner basis of the ideal $I$, one has to check that each ambiguity $w_{kc}w_{jb}w_{ia}$ is
solvable (does not give rise to new relations). Applying replacements $w_{pq}w_{ia}\longrightarrow w_{11}w_{ia}$,  it is not
difficult to check that each ambiguity $w_{kc}w_{jb}w_{ia}$ has normal form $w_{11}w_{11}w_{ia}$. Hence, by the Diamond Lemma,
$\Re$ is a Gr\"{o}bner basis of the ideal $I=(\Re)$ of the free associative algebra $\k \langle \mathcal{W}\rangle$. It is now
obvious that the set $\Re$ is a reduced Gr\"{o}bner basis of the ideal $I$.
This proves parts (1) and (2).

Consider now the Yang-Baxter algebra $\mathfrak{A} = \cA (\k , X\circ Y, r_{X\circ Y})$ of the
Cartesian product  $(X\circ Y, r_{X\circ Y})$.  By definition, $\mathfrak{A}$ is generated by $X\circ Y = \mathcal{W}$ and has
defining relations
which coincide with $\Re(A\circ B)$, see (\ref{eq:Segre1}). Therefore, $\mathfrak{A} = \k \langle \mathcal{W}\rangle/(\Re(A\circ
B)) \cong A \circ B.$
\end{proof}

\subsection{Segre morphisms in the case of permutation idempotent solutions}
\label{Sec:SegreMaps}

Roughly speaking, to introduce an analogue of Segre morphism for Segre products of two quadratic algebras $A\circ B$, one needs a
quadratic algebra
$C$ of a type similar to the type of $A$ and $B$ and an algebra homomorphism $s: C \longrightarrow A \otimes B$, such that the
image of
$s$ is the Segre product $A\circ B$.

We keep the conventions and notation from the previous subsection, so  $(X, r_f)$ and $(Y, r_{\varphi})$ are disjoint permutation
idempotent solutions of the YBE of finite orders $m$ and $n$, respectively,
 $A =\cA(\k ,  X, r_f)$, and  $B =\cA(\k ,  Y, r_{\varphi})$  are the corresponding Yang-Baxter algebras.
We fix enumerations  (\ref{XYenum}) as before and, as in  Convention~\ref{rmk:conventionpreliminary1}, we consider the
degree-lexicographic orders on the free monoids $\langle X\rangle$, and $\langle Y\rangle$ extending these enumerations. The
Segre product of $A\circ B$ has set of one-generators $\mathcal{W}$ ordered as in (\ref{eq:W}) and $(X\circ Y, r_{X\circ Y})$ is
the solution isomorphic to the Cartesian product $(X\times Y, \rho_{X\times Y})$ in Lemma~\ref{pronotation}.

\begin{dfn}
\label{def:Z}
Let $Z=\{z_{11}, z_{12}, \cdots, z_{mn}\}$ be a set of order $mn$, disjoint with $X$ and $Y$, and  define
\[r_\Phi: Z\times Z\longrightarrow Z\times Z,\quad
r_\Phi(z_{jb},z_{ia})=(\Phi(z_{ia}), z_{ia});\quad \Phi\in \Sym(Z),\quad  \Phi(z_{ia})= z_{pq}\;\text{iff}\;
f(x_i)= x_p,\; \varphi(y_a)= y_q.
\]
as the permutation idempotent solution induced canonically from
$(X\circ Y, r_{X\circ Y})$.
\end{dfn}
We adopt the degree-lexicographic order on the free monoid $\langle Z\rangle$ induced by the enumeration of $Z$, where
\[Z= \{z_{11}< z_{12}< \cdots <  z_{mn}\}.\]
\begin{rmk}
\label{remark}
Let  $A_Z= \cA(\k ,  Z, r_{\Phi})$ be the  Yang-Baxter algebra of the permutation solution $(Z,r_{\Phi})$.
Then, by Theorem~\ref{thm:YBalgPermSol},  $A_Z =\k \langle Z\rangle/(\Re(A_Z))$, where the ideal of relations of $A_Z$
is generated by the set $\Re(A_Z)$ consisting of $mn(mn-1)$
quadratic binomial relations
\begin{equation}
\label{eq:viprelationsZ1}
\gamma_{jb,ia}= z_{jb}z_{ia}- z_{11}z_{ia},\quad (j,b)\neq (1, 1),\ 1 \leq i,j\leq m,\ 1 \leq a,b\leq n. \\
\end{equation}
Every relation $\gamma_{jb,ia}$ has leading monomial  $\LM (\gamma_{jb,ia})=z_{jb}z_{ia}$.
\end{rmk}

By definition, $A\circ B$ is a subalgebra of $A\otimes B$. So if an equality holds in $A\circ B$ then it holds in $A\otimes B.$
\begin{lem}
    \label{lem:segremap}
In notation as above, let $(Z, r_\Phi)$ be the permutation idempotent solution of order $mn$ in Definition~\ref{def:Z} and let
$A_Z= \cA(\k ,  Z, r_\Phi)$ be its  YB algebra.
 The  assignment
 \[z_{11} \mapsto x_1\otimes y_1,\quad z_{12} \mapsto x_1\otimes y_2,\quad \cdots,\quad z_{mn} \mapsto x_m\otimes y_n \]
extends to an algebra homomorphism $s_{m, n} : A_Z \longrightarrow A\otimes_{\k } B$.
    \end{lem}
    \begin{proof}
We set
$s_{m,n} (z_{i_1a_1}\cdots z_{i_pa_p}): =(x_{i_1}\circ y_{a_1})\cdots   (x_{i_p}\circ y_{a_p})$, for all
words $z_{i_1a_1}\cdots z_{i_pa_p}\in \langle Z\rangle$ and then extend this
map linearly.
    Note that for each polynomial $\gamma_{jb, ia}\in \Re(A_Z)$  given in (\ref{eq:viprelationsZ1})
one has
    \[s_{n,d}(\gamma_{jb, ia}) = F_{jb,ia}\in \Re,\]
where $\Re$ is the set of relations of the Segre product $A\circ B$ given in Theorem~\ref{thm:rel_segre}, see
(\ref{eq:SegrePBW}).

Since $F_{jb,ia}$ equals identically zero in $A\circ B=\bigoplus_{i \geq 0}  A_i\otimes_{\k } B_i$, which is a subalgebra of
$A\otimes B$,
    it follows that $s_{n,d}(\gamma_{jb, ia}) =F_{jb,ia}= 0$ in  $A\otimes B$.
Therefore the map $s_{m,n}$ agrees
   with the relations of the algebra $A_Z$.
It follows that the map $s_{m, n} : A_Z \longrightarrow A\otimes_{\k } B$
is a well-defined homomorphism of algebras.
\end{proof}

\begin{dfn}
\label{def:segremap}
The map $s_{m, n} : A_Z \longrightarrow A\otimes_{\k } B$ in Lemma~\ref{lem:segremap} is called \emph{the $(m,n)$-Segre map}.
\end{dfn}

\begin{cor}
\label{thm:segremap}
As above, let  $(X, r_f)$, $(Y, r_{\varphi})$  be finite permutation idempotent  solutions on disjoint sets $X = \{x_1, \cdots,
x_m\}$, $Y = \{y_1, \cdots, y_n\}$ and $A=\cA(\k,X,r_f)$, $B=\cA(\k,Y,r_\varphi)$. Let $(Z,r_\Phi)$  be the solution on $Z =
\{z_{11}, \cdots, z_{mn}\}$ in Definition~\ref{def:Z} and  $A_Z= \cA(\k ,  Z, r_\Phi)$.  The image of the Segre map
$s_{m, n} : A_Z \longrightarrow A\tens_\k B$ is the Segre product $A\circ B$. Moreover, the Segre map is an isomorphism of graded
algebras $s_{m, n} : A_Z \longrightarrow A\circ B$.
\end{cor}


\section{Noncommutative differential calculus on $\cA(\k,n)$}\label{secncg}

Because we have proven that $\cA(\k, X, r_f)$ is independent of $f$ up to isomorphism, we have a canonical representative given
by $f=\id$ for the Yang-Baxter algebras associated to the class of permutation idempotent solutions. This is a quadratic algebra
$A=\cA(\k, n)$ with generators $x_1,\cdots,x_n$ where $n=|X|$ and relations from Theorem~\ref{thm:YBalgPermSol} which we write
equivalently as
\[  x_j x_p= x_p^2,\quad  j\ne p,\quad 1\le p\le n.  \]
Here, we add to the algebraic-geometric properties of this algebra in previous sections some first results about their
noncommutative differential geometry.

\subsection{Recap of noncommutative differentials}

Many noncommutative unital algebras $A$ do not admit sufficiently many derivations $A\to A$ to play the role of the classical
notion of partial differentials. Instead, the notion of a derivation on $A$ is naturally generalised to the following data.

\begin{dfn} \label{defcalc} Given a unital algebra $A$ over $\k$, a first order differential calculus means a pair
$(\Omega^1,\extd)$, where
\begin{enumerate}
\item $\Omega^1$ is an $A$-bimodule;
\item $\extd: A\to \Omega^1$ is a derivation in the sense $\extd(ab)=(\extd a)b+a\extd b$ for all $a,b\in A$;
\item The map $A\tens A\to \Omega^1$ sending $a\tens b\mapsto a\extd b$ is surjective.
\end{enumerate}
Here necessarily $\k.1\subseteq \ker \extd$, and $(\Omega^1,\extd)$ is called {\em connected} if $\ker\extd=\k.1$.
\end{dfn}

Given a first order calculus, there is a maximal extension to a differential graded algebra $(\Omega_{max},\extd)$, see
\cite[Lem.~1.32]{BeggsMajid}, with other differential graded algebras $(\Omega,\extd)$ over $A$ with the same $\Omega^1$ a
quotient of this. We recall that $\Omega$ here is a graded algebra with product denoted $\wedge$, $\Omega^0=A$ and $\extd$ is a
graded derivation with $\extd^2=0$.

\begin{rmk} A connected first order calculus always exists, namely there is a universal construction $\Omega^1_{uni}\subset
A\tens A $ defined as the kernel of the product with $\extd_{uni} a=1\tens a - a\tens 1$. Any other first order calculus a
quotient of this by an $A$-sub-bimodule. Also note that first order calculi are similar to the Kahler differential for
commutative algebras and have been used since the 1970s, for example in the works of Connes, Quillen and others.
\end{rmk}

\begin{lem} The quadratic algebra $\cA(\k,n)$ for $n\ge 2$ does not admit any derivations that lower the degree by 1, other than
the zero map.
\end{lem}
\proof Let $D$ be degree lowering $D: A_i \to A_{i-1}$ and obey $D(ab)=aD(b) + D(a)b$ for all
$a, b \in A$. Then $D(x_1) = \alpha, D(x_2) = \beta$ for some $\alpha,\beta\in\k$. Hence $D(x_2x_1)= D(x_2) x_1 + x _2D(x_1) =
\beta x_1 + \alpha x_2.$ But $x_2x_1= x_1^2$ in $A$ and $D(x_1^2) = 2 \alpha x_2$, so $\beta x_1 + \alpha x_2 = 2 \alpha x_2$ and
hence $\alpha = \beta =0$ as $x_1,x_2$ are linearly independent.
 \endproof

We therefore do need a more general concept such as that of a first order differential calculus. For any quadratic algebra with
$n$ generators $x_1,\cdots,x_n$, a sufficient (but not necessary) construction for an $(\Omega^1,\extd)$ that reduces as expected
in the case of $\k[x_1,\cdots,x_n]$ is as follows.

\begin{pro}\label{procalc} Let $A$ be a quadratic algebra on generators $\{x_i\}_{i=1}^n$ and let $\rho: A\to M_n(A)$ be an
algebra map such that
\begin{equation} \label{rho2} \sum_{i,j} r_{ij}( \rho^j{}_{ik}+ x_i\delta_{jk}) =0\quad\forall k\quad {\rm if}\quad
\sum_{i,j}r_{ij}x_ix_j=0, \end{equation}
where  $\rho^j{}_{ik}\in A$ are the matrix entries of $\rho(x_j)$ and $\delta_{jk}$ is the Kronecker $\delta$-function. Then

(1) $\Omega^1$ defined as a free left $A$-module with basis $\extd x_i$ and right module structure
\[   (a\extd x_i)b:= \sum_k (a\rho(b)_{ik})\extd x_k\]
is an $A$-bimodule.

(2) $\extd: A\to \Omega^1$ defined by  $\extd (1)=0, \extd(x_i)=\extd x_i$ extended as a derivation makes $(\Omega^1,\extd)$ into
a first order calculus.

(3) Partial derivatives $\del_i:A\to A$ defined  by
\begin{equation}\label{deli} \extd a= \sum_i (\del_i a)\extd x_i\end{equation}
for all $a\in A$ obey the twisted derivation rule
\begin{equation}\label{del2} \del_i(ab)=\sum_{j}\del_j(a) \rho(b)_{ji}+a\del_i(b)\end{equation}
for all $a,b\in A$.
\end{pro}
\begin{proof}
By definition, $\Omega^1=A\tens V$, where $V$ has basis which we denote $\{\extd x_i\}$, or equivalently $\Omega^1=\Span_A\{\extd
x_1,\cdots,\extd x_n\}$ with the $\extd x_i$ a left basis. The left action is by left multiplication by $A$, so $a(b\extd x_i):=
(ab)\extd x_i$. The right action stated is indeed an action as
\[ ((a\extd x_i).b).c=\sum_j ((a\rho(b)_{ij})\extd x_j).c=\sum_{j,k}(a\rho(b)_{ij}\rho(c)_{jk})\extd x_k=
\sum_{j}(a\rho(bc)_{ij})\extd x_j=(a\extd x_i).(bc).\] By construction, these  form a bimodule. Note that as $A$ here is
quadratic, an algebra map $\rho:A\to M_n(A)$ amounts to $\rho^j:=\rho(x_j)\in M_n(A)$ for $j=1,\cdots,n$ with entries
$\rho^j{}_{ik}\in A$
such that
\begin{equation} \label{rho1} \sum_{i,j} r_{ij} \rho^i \rho^j =0\quad {\rm if}\quad \sum_{i,j}r_{ij}x_ix_j=0,  \end{equation}
and the resulting bimodule is characterised by the bimodule relations
\begin{equation}\label{bimrel}  \extd x_i\ x_j:= \sum_k \rho^j{}_{ik} \extd x_k.\end{equation}

Next, we suppose (\ref{rho2}) and define $\extd: A\to \Omega^1$ as stated. This is well defined as a bimodule derivation since
\[ \extd(x_ix_j)=(\extd x_i)x_j+ x_i\extd x_j=\sum_k (\rho^j{}_{ik} + x_i\delta_{jk} )\extd x_k\]
under our assumption. As the algebra is quadratic, this implies that $\extd$ is well defined on all of $A$.

For the last part, we note that
\[ \extd (a b)=\sum_j \del_j(a)\extd x_j b+ a\sum_i\del_i(b)\extd x_i=\sum_i(\sum_{j}\del_j(a) \rho(b)_{ji}+a\del_i(b))\extd
x_i,\]
which implies the stated property of the $\del_i$ as the $\extd x_i$ are a left basis.
\end{proof}

Note that there is no implication that $\{\extd x_i\}$  are also a right basis, and they will not be in our examples below. In a
geometric context, we could still expect $\Omega^1$ to be right-projective but we do not require or prove this.

\subsection{Differential calculi for $\cA(\k,2)$}\label{secA2calc}

Here, we apply Proposition~\ref{procalc} in the simplest nontrivial case. For calculations, we assume that $\k$ is not
characteristic 2. For $n=2$ we have 2 generators $x=x_1$ and $y=x_2$ with relations
\[ (x-y)x=0,\quad (x-y)y=0\]
which is symmetric between the two generators.  We solve for matrices $\rho^1$ and $\rho^2$ obeying (\ref{rho1})-(\ref{rho2}) and
note that the latter implies the general form
\[ \rho^1=\begin{pmatrix} e & f \\ e+x-y & f\end{pmatrix},\quad \rho^2=\begin{pmatrix} g& h+ y-x\\ g & h\end{pmatrix}\]
for some elements $e,f,g,h\in A$. The former then becomes
\[ \left(\begin{pmatrix} e & f \\ e+x-y & f\end{pmatrix}-\begin{pmatrix} g& h+ y-x\\ g & h\end{pmatrix}\right)\begin{pmatrix} e &
f \\ e+x-y & f\end{pmatrix}=0,\]
\[ \left( \begin{pmatrix} e & f \\ e+x-y & f\end{pmatrix}-\begin{pmatrix} g& h+ y-x\\ g & h\end{pmatrix}\right)\begin{pmatrix} g&
h+ y-x\\ g & h\end{pmatrix}=0.\]
These matrix equations  lead to only four independent equations among the entries, namely
\begin{equation}\label{ef1} (e-g + f-h+z)f=(e-g+f-h+z)g=0,\end{equation}
\begin{equation}\label{ef2} fh-h^2+z h+ (e-g)(h-z)= ge-e^2- ze+ (h-f)(e+z)=0,\end{equation}
where $z:=x-y$ is a shorthand.

The simplest class of solutions of these is to assume that $e,f,g,h\in A^+$, i.e., have all their terms of strictly positive
degree (so each term has a left factor of $x$ or $y$) and that $h-e, f-g$ are each divisible by $z$ as a right factor. Then the
first two  equations are automatic as is the sum of the latter two. All that remains is their difference, which reduces to
\[ (h+e-f-g)z=0.\]
These requirements do not have a unique solution, but the lowest degree solution is to take $e,f,g,h$ to be degree 1 with
\[ f=\lambda x+ (1-\lambda)y,\quad g=\mu x+ (1-\mu)y,\quad e=\alpha x+ (1-\alpha)y,\quad h=\beta x+ (1-\beta)y\]
for parameters $\lambda,\mu,\alpha,\beta\in \k$. The result can be written compactly as
\[ \rho^i{}_{jk}=y+(\epsilon_{ij}\delta_{ik}+C_{ik})z,\quad C=\begin{pmatrix} \alpha & \lambda \\ \mu &\beta\end{pmatrix},\]
where $\epsilon_{ij}$ is the antisymmetric function with $\epsilon_{12}=1=-\epsilon_{21}$ and other entries zero.  The
 bimodule relations are
\[ \extd x\ x=(y+\alpha z)\extd x+ (y+\lambda z)\extd y,\quad \extd x\ y=(y+\mu z)\extd x+ (y+(\beta-1) z)\extd y,\]
\[ \extd y\ x= (y+(\alpha+1) z)\extd x+ (y+\lambda z)\extd y,\quad \extd y\ y =(y+\mu z)\extd x+ (y+\beta z)\extd y.\]
Note that these relations are symmetric between $x,y$ iff
\[ \beta=1-\alpha,\quad \mu=1-\lambda,\]
so we have a 2-parameter family of these. Also note that $\extd z\ x=-z\extd x$  and hence that $(\extd x-\extd y)a=0$ for any
$a\in A$ with all terms of degree $\ge 2$. Hence $\extd x,\extd y$ are never a right-basis.

We also have $(\Omega_{max},\extd)$ given by applying $\extd$ to the bimodule relations with  $\extd^2=0$ and the graded
derivation rule. This is given in degree 2 by the relations
\[  \lambda \extd x \wedge\extd y+(1-\alpha)\extd y\wedge\extd x+(\alpha+1)\extd x\wedge\extd x+(1-\lambda)\extd y\wedge\extd
y=0, \]
\[  \beta\extd x \wedge\extd y+(1-\mu)\extd y\wedge\extd x+\mu\extd x\wedge\extd x+(2-\beta)\extd y\wedge\extd y=0. \]

\begin{ex}\label{calcsym} For a concrete ($x-y$ symmetric) example, we can take
\[ f=e=x, \quad g=h=y,\quad C=\begin{pmatrix}1 & 1\\ 0& 0\end{pmatrix}. \]
The bimodule relations are then
\[\extd x\ x=x(\extd x+\extd y) ,\quad \extd y\ x=(2x-y)\extd x + x\extd y,\quad \extd x\ y=y\extd x+ (2y-x)\extd y,\quad \extd
y\ y=y(\extd x+\extd y).\]
 Next, by iterating the bimodule relations, one finds
\[ \extd x_i \begin{cases} x^m \\ y^m\end{cases}= 2^{m-2}\begin{cases}(3 x^m-y^m)\extd x+ 2 x^m \extd y \\
2 y^m \extd x+ (3y^m-x^m)\extd y\end{cases},\quad \forall m\ge 2\]
independently of $i$. The partials can then be computed by
iterating the Leibniz rule for $\extd$ using these relations (or from (\ref{del2})), to find
\[\del_i (1)=0,\quad \del_i (x_j)=\delta_{ij},\quad \del_i (x_i^m)=(3\ 2^{m-1}-1)x_i^{m-1}- (2^{m-2}-1)x_{\bar i}^{m-1},\quad
\del_i (x_{\bar i}^m)=(2^{m-1}-1)x_{\bar i}^{m-1},\]
for $m\ge 2$, where $x_{\bar i}$ denotes the other generator from $x_i$. As $A=\k1\oplus\k[x]^+\oplus\k[y]^+$, this specifies the
linear maps $\del_i$. They lower degree  by 1  but are not derivations. From these formulae, it follows easily that
$\del_i(a)=0$ implies $a\in \k 1$, and hence that the calculus is connected. The relations of $(\Omega_{max},\extd)$ are
\[ \extd x\wedge\extd y=-2\extd x\wedge\extd x,\quad \extd y\wedge\extd x= -2\extd y\wedge\extd y.\]
\end{ex}

\subsection{Monoid-graded differential calculus for $\cA(\k,n)$}\label{secmoncalc}

 As $\cA(\k,n)$ is the algebra of a monoid $S=S(X,r_{\id})$, it is necessarily a cocommutative bialgebra and one can ask for
 translation-covariant calculi with respect to this. Explicitly, the comultiplication and counit of $A=\k S$ for a monoid $S$ are
 $\Delta(s)=s\tens s$ and $\epsilon(s)=1$ for all $s\in S$ and translation-covariance amounts to $(\Omega^1,\extd)$ admitting the
 diagonal grading where the grades of $s,\extd s$ are $s$. The prescription for such first order calculi in \cite[Thm.
 1.47]{BeggsMajid} in the group algebra case can still be applied, namely we start with a right action of $S$ on a vector space
 $V$ and an element $\theta\in V$ and define $\tilde\Omega^1=A\tens V$ as a free left module. The left action, right action and
 $\extd$ are
 \[s.(t\tens v)=st\tens v,\quad  (t\tens v).s=ts\tens v.s,\quad \extd s=s\tens (\theta.s-\theta),\quad\forall s,t\in S,\]
 which gives a generalised calculus $(\tilde\Omega^1,\extd)$ in the sense of dropping condition (3) of Definition~\ref{defcalc}.
 We then define an actual calculus $\Omega^1\subseteq\tilde\Omega^1$ as the image of the map $a\tens b\mapsto a\extd b$ for
 $a,b\in A$. In our case, omitting $\tens$ and giving a more explicit treatment with $V=\k^n$  as row vectors, we have the
 following.

(a) Natural $n\times n$ matrix representations of $\cA(\k,n)$ that separate $x_i$ are of the form
\[ x_i\mapsto \rho_i=\xi\tens u_i,\quad u_i\cdot\xi=1+\mu,\]
where $\xi$ is an $n\times 1$ column vector and $u_i$ are distinct $1\times n$ row vectors, all with entries in $\k$, and
$\mu\in\k$. We let $\theta$ be another row vector and suppose that
\[ e_i=\theta\cdot(\rho_i-\id)=(\theta\cdot\xi) u_i-\theta\]
are linearly independent, where $\id$ is the $n\times n$ identity matrix. We then define coefficients $\gamma_{ij}\in \k$ by
\[ e_k\cdot\rho_i=\theta\cdot(\rho_k-\id)\cdot\rho_i=\theta\cdot(\rho_i-\id)\cdot\rho_i=(\theta\cdot\xi)\mu
u_i=\sum_j\gamma_{ij} e_j\]
independently of $k$. We arrive at a generalised calculus as a free module with left basis $\{e_i\}$,
\[ \tilde \Omega^1=\Span_{A}\{e_i\},\quad \extd x_i=x_i e_i, \quad e_k x_i=x_i \sum_j \gamma_{ij} e_j \]
independently of $k$. We specify $\extd$ and the bimodule via its relations.

(b) If we are in the generic situation where $v_i:=\sum_{j\ne i}\gamma_{ij}e_j\ne 0$ for all $i=1,\cdots,n$, the image
subcalculus has the form
\[ \Omega^1=\oplus_i \Span_{\k[x_i]^+}\{e_i, x_i v_i\} \]
with bimodule relations such as
\[\extd x_k\  x_i=\gamma_{ii} x_i \extd x_i + x_i^2v_i\]
independently of $k$. The left hand side is also $\extd x_i^2-x_i\extd x_i$,   which expresses $x^2_iv_i$  in terms of elements
of the form $a\extd b$. Note that $z_{ij}:=x_i-x_j$ acting from the left annihilates all of $\Omega^1$, so the $\{\extd x_i\}$
are not a left basis. Also observe that the bimodule relations are indeed compatible with the grading, for example  $\extd x_k\
x_i$ has grade $x_kx_i=x_i^2$ independently of $k$.  Both features are very different from our previous construction via
Proposition~\ref{procalc}. Indeed, none of the 4-parameter calculi in Section~\ref{secA2calc} on $\cA(\k,2)$ are compatible with
diagonal grading by the monoid.

\subsection{FRT bialgebra and covariance of $\cA(\k,X,r_f)$}

If $(X,r)$ is a braided set, we let $V=\k X$ and extend $r$ by linearity to a map $\Psi:V\tens V\to V\tens V$, where we identify
$V\tens V=\k X\times X$ in the obvious way. If $\{x_i\}$ is an enumeration of the elements of $X$ then this takes the form
\begin{equation}\label{PsiR} \Psi(x_i\tens x_j)=\sum_{a,b}x_b\tens x_a R^a{}_i{}^b{}_j= f(x_j)\tens x_j= \sum_{a,b}f_{jb}x_b\tens
\delta_{ja}x_a\end{equation}
in the permutation idempotent case, where $R^a{}_i{}^b{}_j\in \k$ is the corresponding R-matrix in the conventions of
\cite{MajidQG} and we use the specific form of $r=r_f$ with $f(x_j)=\sum_b f_{jb}x_b$ for coefficients $f_{jb}\in \k$. Comparing,
we see that
\[ R^a{}_i{}^b{}_j=f_{jb}\delta_{ja}.\]

Next, associated to an R-matrix one can define a quadratic algebra $\check V(R)$ with generators $x_i$ and relations
\begin{equation}\label{xxR} x_ix_j=x_bx_a R^a{}_i{}^b{}_j= x_bx_af_{ab}\delta_{aj}=f(x_j)x_j,\quad\forall i,j\end{equation}
in the permutation idempotent case. We obtain here $\cA(\k,X,r_f)$ or $\cA(\k,n)$ for $f=\id$.  Moreover, this is necessarily a
comodule algebra via the algebra map
\begin{equation}\label{coact} x_i\mapsto \sum_a x_a\tens t^a{}_j\end{equation}
under the FRT bialgebra with $n^2$ generators $\{t^i{}_j\}$ and FRT relations\cite{FRT}
\[ \sum_{a,b}R^i{}_a{}^k{}_bt^a{}_j t^b{}_l=\sum_{a,b} t^k{}_b t^i{}_a R^a{}_j{}^b{}_l.\]
These translate in our case to the relations
\[ f_{ki} (\sum_a t^a{}_j)t^k{}_l= \sum_b t^k{}_b f_{lb} t^i{}_l\]
for all $i,j,k,l$ or
\begin{equation}\label{Bnrelns} \delta_{ki} (\sum_a t^a{}_j)t^i{}_l=t^k{}_l t^i{}_l\end{equation}
when $f=\id$. The coalgebra on the generators in all cases is
\[ \Delta(t^i{}_j)=\sum_a t^i{}_a\tens t^a{}_j,\quad \epsilon(t^i{}_j)=\delta_{ij}.\]
We denote this bialgebra by $\cB(\k,X,r)$ for any linearised braided set, and in our case of interest for $f=\id$ by $\cB(\k,n)$.
For $n=2$, the latter amounts to the relations
\[ t^{\bar i}{}_l t^i{}_l=0,\quad (t^i{}_l)^2=(\sum_a t^a{}_{\bar l})t^i{}_l\]
for all $i,l$, where $\bar i$ denotes the other index value to $i$. It follows, but is a useful check that to verify directly,
that $\cA(\k,2)$ is covariant under $\cB(\k,2)$, i.e. that the coaction (\ref{coact}) extends as an algebra homomorphism
$\cA(\k,2)\to \cA(\k,2)\tens \cB(\k,2)$.

\begin{pro}\label{procalccov} Suppose in Proposition~\ref{procalc} that the entries of $\rho$ have degree 1 so that
$\rho^j{}_{ik}=\sum_m \rho^j{}_{mik}x_m$ and suppose that $A$ is a comodule algebra under $A(R)$ via (\ref{coact}). Then
$(\Omega^1,\extd)$ is covariant under the coaction iff
\[ \sum_{a,b} \rho^b{}_{kai}t^a{}_j t^b{}_l =\sum_{a,b}t^k{}_b t^i{}_a \rho^l{}_{bja}\]
for all $i,j,k,l$.\end{pro}
\begin{proof} The coaction, for the calculus to be covariant, is required to extend to products in such a way that $\extd$ is a
comodule map. Hence the coaction, and then the bimodule relations, applied to the left hand side of $\extd x_j\
x_l=\sum_{a,b}\rho^l{}_{bja}x_b\extd x_a$ give
\[\extd x_j\ x_l\mapsto\sum_{a,b}\extd x_a\ x_b\tens t^a{}_j t^b{}_l=\sum_{a,b,k,i}x_k\extd x_i\tens \rho^b{}_{kai}t^a{}_j
t^b{}_l,\]
while the coaction similary applied to the right hand side of the same equation gives $\sum_{k,i,a,b}x_k\extd x_i\tens t^k{}_b
t^i{}_a \rho^l{}_{bja}$. Since the $\{\extd x_i\}$ are a left basis, we require the condition stated. Conversely, if this holds
then we can extend the coaction to $\Omega^1$ in this way.
\end{proof}

\begin{rmk}\label{remR'} These relations are similar to but not necessarily the same as the FRT relations.  In fact the R-matrix
theory here is part of a general construction\cite{MajidQG} of quantum-braided planes $\check V(R',R)$ associated to a pair of
compatible matrices, where $R'$ is used to define the relations in place of $R$ in (\ref{xxR}). To simply have an $A(R)$-comodule
algebra, one needs some mixed YBE conditions with an outer $R$ on each side replaced by $R'$, so one can simply take $R'=R$ as we
have done above. Also, the category of comodules of $A(R)$ is prebraided (by which we mean that the braiding generated as above
by $R$ need not be invertible) and there are further conditions\cite[Thm.~10.2.1]{MajidQG} for $\check V(R',R)$ to be a Hopf
algebra in this prebraided category, and further conditions (equations (10.61) in the same work) which allow for a canonical
first order calculus. These all apply in the involutive or $q$-Hecke cases (with $R'\propto R$ and $\rho^j{}_{mik}$ also given by
$R$), but not in the case of $\Psi$ idempotent as here.

In particular, one can check by direct calculation that none of the 4-parameter  moduli calculi  $\cA(\k,2)$ in
Section~\ref{secA2calc} nor the monoid-graded calculus in Section~\ref{secmoncalc} for $n=2$ are covariant under $\cB(\k,2)$.
Hence, the construction of differential calculi on $\cA(\k,n)$ covariant under $\cB(\k,n)$ remains open.
\end{rmk}

\subsection{Fermionic YB algebra $\Lambda(\k,X,r_f)$}

In the familiar case of $R$ involutive or $q$-Hecke, one has a further `fermionic'  quadratic algebra which deforms the exterior
rather than the symmetric algebra generated by the $x_i$, and which is again a comodule algebra under $A(R)$, see \cite{MajidQG}.
In the general set-up of $\check V(R',R)$ mentioned in Remark~\ref{remR'}, if $R$ is such that $\Psi$ defined as in the first
expression of (\ref{PsiR}) is idempotent, the required choices for such a `fermionic' braided Hopf algebra are to use $-R$ in
place of $R$ to define the new (pre)braiding $\Psi$ and $R'=R+P$ to define the relations of the quadratic algebra, where
$P^i{}_j{}^k{}_l=\delta^i{}_l\delta^k{}_j$ is the matrix for the flip $V\tens V\to V\tens V$. Calling the quadratic algebra
generators now $\theta_i$ in place of $x_i$, the relations of this `fermionic' YB algebra are
$\theta_i\theta_j=\sum_{a,b}\theta_b \theta_a R^a{}_i{}^b{}_j + \theta_i\theta_j$ and hence
\[ \sum_{a,b}\theta_b \theta_a R^a{}_i{}^b{}_j =0,\quad f(\theta_j)\theta_j=0\]
for all $i,j$ in the case of $R$ linearising $r_f$. Here, $\Psi$ is again not invertible, now being given by
\[ \Psi(\theta_i\tens\theta_j)=-\sum_{a,b}\theta_b\tens \theta_a R^a{}_i{}^b{}_j =-f(\theta_j)\tens \theta_j\]
in the case of $R$ obtained from $r_f$. By \cite[Thm.~10.2.1]{MajidQG}, we have a braided coalgebra and antipode which on
generators has the form
\[ \Delta(\theta_i)=\theta_i\tens 1+1\tens\theta_i,\quad \epsilon(\theta_i)=0,\quad S\theta_i=-\theta_i,\]
where $\Delta$ extends to products provided we allow for $\Psi$ to exchange tensor products. It is easy enough to check this
directly:
\begin{align*} \sum_{a,b}\Delta(\theta_b \theta_a) R^a{}_i{}^b{}_j&=\sum_{a,b}(\theta_b\tens 1+1\tens\theta_b)\underline\cdot
(\theta_a\tens 1+1\tens\theta_a)R^a{}_i{}^b{}_j\\
&=\sum_{a,b}(\theta_b\theta_a\tens 1+1\tens\theta_b\theta_a+ \theta_b\tens\theta_a+
\Psi(\theta_b\tens\theta_a))R^a{}_i{}^b{}_j,\end{align*}
which vanishes precisely when $-\Psi$ is idempotent. The antipode likewise extends to products using $\Psi$. This construction
applies, in particular, for any idempotent solution $(X,r)$ of the set-theoretic braid relations and in that context we call this
$\Lambda(\k,X,r)$, to reflect the skew-symmetric character. Clearly, for $f=\id$, this is just the free algebra on $\{\theta_i\}$
modulo the relations $\theta_i^2=0$
and  $\Psi(\theta_i\tens\theta_j)=-\theta_j\tens\theta_j$.


\end{document}